\documentclass[onecolumn,draftclsnofoot]{IEEEtran}
\usepackage{amsmath,amsfonts,amsthm, cuted} 
\usepackage{comment}
\usepackage{setspace}
\usepackage{graphicx}
\usepackage{lipsum}
\usepackage{wasysym}
\usepackage{subcaption}
\usepackage{resizegather}
\usepackage{color}
\usepackage{enumitem}
\usepackage{float}
\usepackage[linesnumbered,ruled]{algorithm2e}

\newcommand{\ud}[1]{_\mathrm{#1}}
\newcommand{\up}[1]{^\mathrm{#1}}
\graphicspath{{figures/}}
\newtheorem{lemma}{Lemma}

\newtheorem{theorem}{Theorem}
\newtheorem{proposition}{Proposition}

\newcommand{\bd}{\mathbf}

\title{A Convex Cycle-based Degradation Model for Battery Energy Storage Planning and Operation}
\author{
	Yuanyuan Shi,
	Bolun Xu,
	Yushi Tan,
	and Baosen Zhang
	
	\thanks{The authors are with the Department of Electrical Engineering,
		University of Washington,
		Seattle, Washington 98125,
		(e-mail:\{yyshi, xubolun, ystan, zhangbao\}@uw.edu)}
}

\begin{document}
\maketitle
\begin{abstract}
		\label{sec:abstract}
A vital aspect in energy storage planning and
operation is to accurately model its operational cost, which
mainly comes from the battery cell degradation. Battery degradation
can be viewed as a complex material fatigue process that based
on stress cycles. Rainflow algorithm is a popular way for cycle
identification in material fatigue process, and has been extensively
used in battery degradation assessment. However, the rainflow
algorithm does not have a closed form, which makes the major
difficulty to include it in optimization. In this paper, we prove the
rainflow cycle-based cost is convex. Convexity enables the proposed
degradation model to be incorporated in different battery
optimization problems and guarantees the solution quality. We
provide a subgradient algorithm to solve the problem. A case
study on PJM regulation market demonstrates the effectiveness
of the proposed degradation model in maximizing the battery operating profits as well as extending its lifetime.
\end{abstract}

\section{Introduction}
\label{sec:intro}
Battery energy storage (BES) is becoming an essential resource in energy systems with high renewable penetrations. Applications of BES include peak shaving~\cite{dunn2011electrical}, frequency regulation~\cite{shi2016leveraging}, demand response~\cite{li2011optimal,li2016optimal}, renewable integration~\cite{bitar2011role}, grid transmission and distribution support~\cite{pandzic2015near}, and many others. Although each application has different objectives and constraints, a common theme among them is that accurately accounting for the \emph{operational cost of batteries} is of critical importance.

Battery operating cost is mainly caused by the degradation effect of repeated charging and discharging~\cite{zakeri2015electrical}. Battery cells typically reach end-of-life (EoL) if their capacity degrades beyond a certain minimum threshold~(for example, 80\% of original capacity)~\cite{vetter2005ageing}, after which cells can no longer perform as expected. The resulting cell replacement cost represents the predominate operating cost of batteries, because of high manufacturing price for most electrochemical battery cells~\cite{nykvist2015rapidly}. The goal of this paper is to present an \emph{electrochemically accurate} and \emph{tractable} battery degradation model that can be used in multiple applications.

To account for battery degradation, two main classes of models have been considered. The first kind of degradation model is based on battery charging/discharging \emph{power}~\cite{li2011optimal,shi2016leveraging,guo2013electricity}. For example, \cite{shi2016leveraging} used a linear degradation cost while considering peak shaving and regulation services, and~\cite{li2011optimal} adopts a general convex cost for battery demand response. This power-based model decouples the degradation cost in time, since the costs at each of the time steps are independent. It is easy to be incorporated in battery optimization.
However, the major concern of the power-based model is about its accuracy in capturing the actual degradation cost. For example, a Lithium Nickel Manganese Cobalt Oxide (NMC) battery has \emph{ten} times more degradation when operated at near 100\% cycle depth of discharge (DoD) compared to operated at 10\% DoD for the \emph{same} amount of charged power~\cite{ecker2014calendar}. Power-based degradation model fails to capture such cumulative effect of battery cell aging and may severely deviate the operation from optimal.

Therefore, we consider a more accurate model to account for battery degradation, which is called \emph{cycle-based} model. This model is based on fundamental battery aging mechanism. It views the battery degradation process as a complex material fatigue process~\cite{xu2016modeling} that based on stress cycles. Battery aging at each time is not independent, but closely related to accumulation of previous charging and discharging. 

To accurately identify cycles in an irregular battery state of charge~(SoC) profile, we adopt the rainflow cycle counting method. Rainflow algorithm has been extensively used for cycle identification in material fatigue analysis~\cite{rychlik1987new,downing1982simple,amzallag1994standardization} as well as battery degradation model~\cite{muenzel2015multi,dragivcevic2014capacity,musallam2012efficient}. Fig. \ref{fig:intro_rainflow} gives an example of rainflow cycle identification results of a battery SoC profile. The total battery degradation cost is modeled as the sum of degradation from all cycles. 
\begin{figure}[!ht]
		\centering
		\includegraphics[width=4 in]{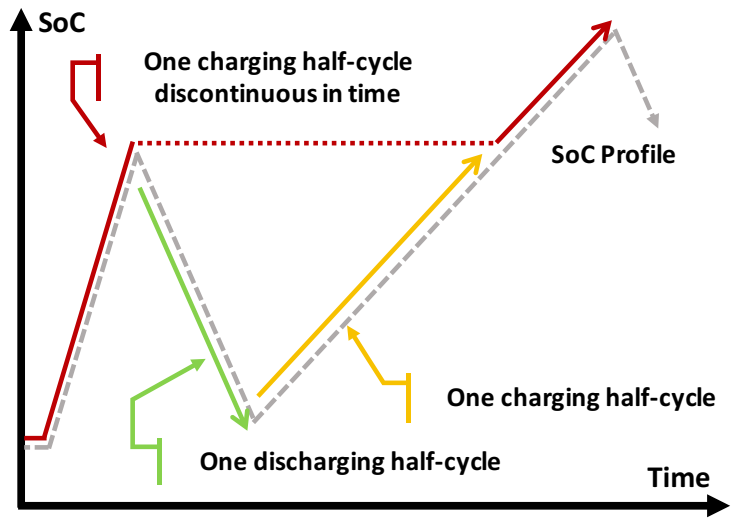}
		\caption{Rainflow cycle counting example. There are three half cycles, one charging half-cycle (the red one) is discontinuous in time. The green discharging half-cycle and yellow charging half-cycle are of the same cycle depth.}
		\label{fig:intro_rainflow}
\end{figure}

Cycle-based degradation model can accurately model the fundamental battery aging. However, it is considerably more cumbersome and computationally difficult than the power-based cost model. The rainflow algorithm, despite its wide array of applications, is a procedure that does not have a closed form. This makes it difficult to be incorporated in optimization problems. Most previous works (see~\cite{xu2016modeling} and references within) apply the rainflow algorithm for a posterior battery degradation assessment with fixed operating strategies.

Instead of posterior degradation assessment, we want to incorporate rainflow algorithm in battery optimization. Several efforts have been made by simplifying the algorithm procedure, either simplifying the rainflow counting procedure~\cite{he2016optimal,abdulla2016optimal}, or using more trivial cycle definition~\cite{koller2013defining,tran2013energy}. These approaches improve the problem solvability, however the simplifications introduce additional errors and sacrifices the optimality of solutions.

The main contribution of our paper is two-folded:
\begin{enumerate}
\item  We prove the rainflow cycle-based degradation model is \emph{convex}. Convexity enables the proposed cost model to be used in various battery optimization problems and guarantees the solution quality.
\item We provide a subgradient algorithm to solve the battery optimization problems.
\end{enumerate}

The proposed degradation model and subgradient solver algorithm have a broad application scope for battery operation optimization. We implement it for a case study on PJM regulation market, and verify this model accurately captures the battery cell aging, significantly improve the operational revenue (up to 30\%) and almost doubled the expected BES lifetime compared with previous degradation models.

The rest of the paper is organized as follows. Section \ref{sec:rainflow} describes the proposed rainflow cycle-based degradation model. Section \ref{sec:convexity} sketches the convexity proof. Section \ref{sec:subgradient} gives a subgradient algorithm. We provide a case study in Section \ref{sec:case} using real data from PJM regulation market, and demonstrate the effectiveness of the proposed degradation model in maximizing the BES operating profits as well as extending battery lifetime. Finally, Section \ref{sec:con} concludes the paper and outlines directions for future work.

\section{Model}
\label{sec:rainflow}
\subsection{Battery degradation model}
\label{sec:II-A}
We consider a battery operating in power markets and providing certain profitable service. We use $c(t)$ to denote battery charging and $d(t)$ for discharging at time $t$, where $t=1,2,...,T$. The battery SoC is $\mathbf{s} \in \mathbb{R}^T$, which is an accumulation of charged/discharged power. In Section \ref{sec:subgradient}, we will describe in detail how we model the battery operation and constraints. Here, we focus on degradation model.

Suppose the operation revenue from battery application is $R(\cdot)$, which is a function of battery power output. The battery degradation is captured by rainflow cycle-based degradation model. Therefore, a utility function $U(\cdot)$ that quantifies the profits BES owner obtains over $t \in [1, T]$ is,
\begin{align}
U = \underset{\text{}}{R(\mathbf{c}, \mathbf{d})} - \underset{\text{}}{\lambda\up{r} \sum_{i=1}^{N} \Phi(d_i)} \,,\label{eq:utility}
\end{align}
where $[d_1, d_2, ..., d_N]$ are cycle depths. We use rainflow algorithm for cycle identification, which is defined as below,
\begin{align}
[d_1, d_2, ..., d_N] & = \bd{Rainflow} (s(1),s(2), \dotsc, s(T))\,,
\label{eq:rainflow}
\end{align}

The second term in Eq. \eqref{eq:utility} together with Eq. \eqref{eq:rainflow} constitute the proposed \emph{rainflow cycle-based degradation model}, where $[d_1, d_2, ..., d_N]$ are cycle depths identified by rainflow algorithm. $\Phi(\cdot)$ is the DoD stress function, which defines the degradation of one cycle under reference condition.  $\sum_{i=1}^{N} \Phi(d_i)$ calculates the total degradation by summing over all cycles, and $\lambda\up{r}$ is the battery cell replacement price. It's worth to point out that DoD stress function can be fitted with test data of different types of batteries, which offers the proposed degradation model strong generalization ability.
%

\subsection{Rainflow cycle counting algorithm}
\label{sec:rainflow2}
The process of reducing a strain/time history into a number of smaller cycles is termed as cycle counting. Rainflow algorithm is the most popular cycle counting algorithm used for analyzing fatigue data~\cite{rychlik1987new,downing1982simple,amzallag1994standardization}. When used for battery life assessment, it takes a time series of battery's state of charge (SoC) as input, and identifies the depth of all cycles contained in this series. 

Fig. \ref{algorithm_rainflow} gives an example on the rainflow algorithm implementation for cycle counting of a battery SoC profile, and the cycle analysis results are reported in Table \ref{Sec2:T1}. The standard rainflow algorithm procedures based on paper~\cite{rychlik1987new} are given in Algorithm \ref{algorithm:rainflow}. 
\begin{figure}[!t]
	\centering
	\begin{subfigure}[b]{0.48 \columnwidth}
		\centering
		\includegraphics[width= \columnwidth, height= 0.9 \columnwidth]{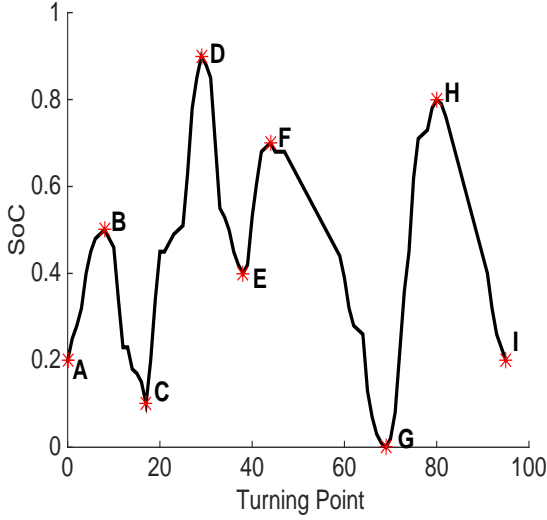}
		\caption[Network2]%
		{{\small Battery SoC time history. Red stars mark the local maximum and local minimum points}}
		\label{fig_soc}
	\end{subfigure}
	\hfill
	\begin{subfigure}[b]{0.48 \columnwidth}
		\centering
		\includegraphics[width=\columnwidth, height= 0.9 \columnwidth]{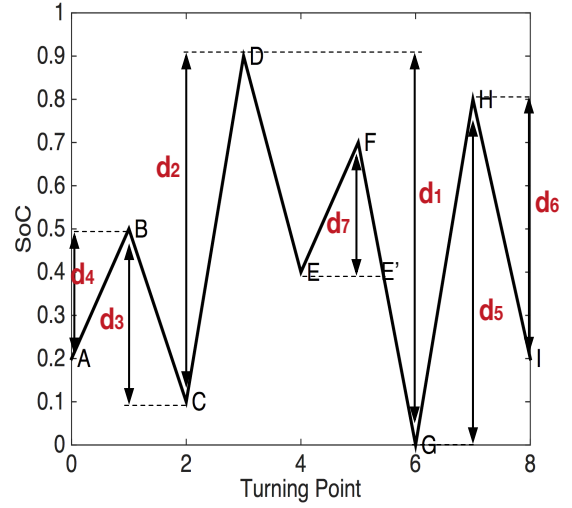}
		\caption[]%
		{{\small Rainflow cycle counting results, based on extracted local maximum and minimum points}}
		\label{fig_turning}
	\end{subfigure}

	\caption[]
	{\small Rainflow cycle counting algorithm procedures}
	\label{algorithm_rainflow}
\end{figure}
\begin{table}[!t]
	\centering
	\begin{tabular}{|c|c|c|c|c|c|c|c|}\hline
		Path & A-B & B-C & C-D & D-G & E-F-E' & G-H & H-I \\
		\hline
		SoC range &0.3 & 0.4 & 0.8 & 0.9 & 0.3 & 0.8 & 0.6 \\
		\hline
		Cycle & half & half & half & half & full & half & half \\
		\hline
	\end{tabular}
	\caption{Rainflow cycle counting results of the SoC profile in Fig. \ref{fig_soc}. The ``SoC range'' is calculated as the absolute difference between the cycle starting and ending SoC.}
	\label{Sec2:T1}
\end{table}

\begin{algorithm}
	\SetKwInOut{Input}{Input}
	\SetKwInOut{Output}{Output}
	\SetKwComment{Comment}{$\triangleright$\ }{}
	\Input{Battery SoC profiles, $\bd s \in \mathbb{R}^{T}$}
	\Output{Cycle counting results: $d_1, d_2, ...d_{N}$}
	
	Reduce the time history to a sequence of turning points -- local maximum and local minimum.
	
	Find the global maximum and global minimum, counted as a half cycle.
	
	If the global maximum happens first: \linebreak
	a. Half cycles are counted between the global maximum and the most negative minimum occurs before it, the most positive maximum occurring prior to this minimum, and so on to the beginning of the history.\linebreak
	b. Half cycles are also counted after the global minimum in the history and terminates at the most positive maximum occurring subsequently, the most negative minimum occurring after this maximum, and so on to the end of the history. \linebreak
	c. Remains are full cycles.
	
	If the global minimum happens first: \linebreak
	Adjust step 3a, 3b, 3c accordingly to pair up sequential most positive maximum and most negative minimum that happen prior to the global minimum, or after the global maximum. Remains are small full cycles.
	
	\caption{Rainflow Cycle Counting Algorithm}
	\label{algorithm:rainflow}
\end{algorithm}
Following the procedures in Algorithm \ref{algorithm:rainflow}, we count cycles in Fig. \ref{fig_soc}. First, all local extremes points of the SoC profile are extracted and arranged as Fig. \ref{fig_turning}. Then the deepest half cycle D-G is identified. Next, half cycles before D are identified, which are C-D($d_2$), B-C ($d_3$), A-B ($d_4$); also half cycles after G are counted, which are G-H ($d_5$) and H-I ($d_6$). Finally, there is one remaining full cycle E-F-E' with depth $d_7$, which could also be viewed as one charging half cycle plus one discharging half cycle with equal depth. 
\subsection{DoD stress function}
\label{sec:dod}
DoD stress function $\Phi(\cdot)$ is a critical part of the degradation model since it captures the cell aging caused by one cycle under reference conditions~\cite{xu2016modeling}. Different types of batteries may have different stress function forms. Some commonly used DoD stress functions are given below.

(1) Linear DoD stress model \cite{shi2016leveraging}, 
\begin{align*}
\Phi(d) = k_1 d\,,
\end{align*}
This linear DoD stress function is equivalent to the linear power-based degradation model. It is simple and suitable under some conditions. However, lab tests show that cycle DoD has a highly nonlinear impact on degradation under most conditions. The following two types of nonlinear DoD stress models are commonly used in literature.

(2) Exponential DoD stress model \cite{millner2010modeling}, 
\begin{align*}
\Phi(d) = k_2 d e^{k_3 d}\,,
\end{align*}

(3) Polynomial DoD stress model \cite{koller2013defining}, where
\begin{equation}
\Phi(d) =  k_4 d^{k_5}\,,
\label{eq:polynomial_dod}
\end{equation}
where all $k$'s are model coefficients, which could be estimated by fitting battery cycling aging test data. In paper, the DoD stress function is assumed to be a convex function, where a higher cycle DoD leads to a more severe damage.

\section{Convexity}
\label{sec:convexity}
The major difficulty of incorporating the rainflow cycle-based degradation model to optimization is that the rainflow algorithm does not have a closed form. It could be solved by some computer numerical methods, however the computational complexity is high and there is no guarantee for the solution quality. We prove that the rainflow cycle-based degradation cost $f(\mathbf{s})$, given a convex DoD stress function, is \emph{convex} in terms of $\mathbf{s}$. It helps to overcome the difficulty of incorporating the rainflow algorithm to battery optimization. We provide a sketch of the convexity proof below. A detailed version of the proof is given in Section \ref{sec:proof}.

\begin{theorem}{Rainflow cycle cost is convex}
	
	The rainflow cycle-based battery degradation model,
	\begin{align*}
	f(\mathbf{s}) = \sum_{i=1}^{N} \Phi(d_i),\ \ \text{where\ } [d_1, d_2,...,d_{N}] = \bd {Rainflow} (\mathbf{s})\,,
	\end{align*} 
	is convex. That is, $\forall \mathbf{s}_1, \mathbf{s}_2 \in \mathbb{R}^{T}$,
	\begin{small}
	\begin{equation}
		f\left(\lambda \mathbf{s_1} + (1-\lambda) \mathbf{s_2}\right) \leq \lambda f( \mathbf{s_1}) + (1-\lambda) f(\mathbf{s_2}), \forall \lambda \in [0,1]
	\end{equation}
	\end{small}
	\label{theo1}
\end{theorem}

This above theorem is intuitively pleasing. Consider two SoC series $\mathbf{s_1}$ and $\mathbf{s_2}$, if they change in different directions, the two signals can cancel each other out partially so that the cost of left side is less than the right side. When $\mathbf{s_1}$ changes in the same direction as $\mathbf{s_2}$ for all time steps, the equality holds. 

We prove Theorem \ref{theo1} by induction. It contains three steps, 1) unit step function decompostion of battery SoC, 2) the initial condition proof and 3) the induction relation proof from K to K+1 step. Here we sketch a proof for the conveixty. 
	
\subsection{Unit step function decompostion}
First, we introduce the step function decompostion of SoC signal. We notice that, any SoC series $\mathbf{s}$, could be written out as a finite sum of step functions, where
	\begin{align}
	\mathbf{s} = \sum_{i =1}^{T} P_i u(t-i) \,, \label{eq:dep_s1}
	\end{align}
	where $u(t-i)$ is a unit step function with a jump at time $i$, and $P_i$ is the jump amplitude.
	\begin{align}
	u(t-i) = 
	\begin{cases}
	1& \text{t $\geq$ i}\\
	0 & \text{otherwise}
	\end{cases} \ \ \forall t = 1,2,...,T\,,
	\end{align}
	
	For notation convenience, we use $U_i$ to denote $u(t-i)$. Fig. \ref{fig:discreted} gives an example of step function decomposition results of $\mathbf{s_1}$.
	\begin{figure}[!ht]
		\centering
		\includegraphics[width=0.8 \columnwidth]{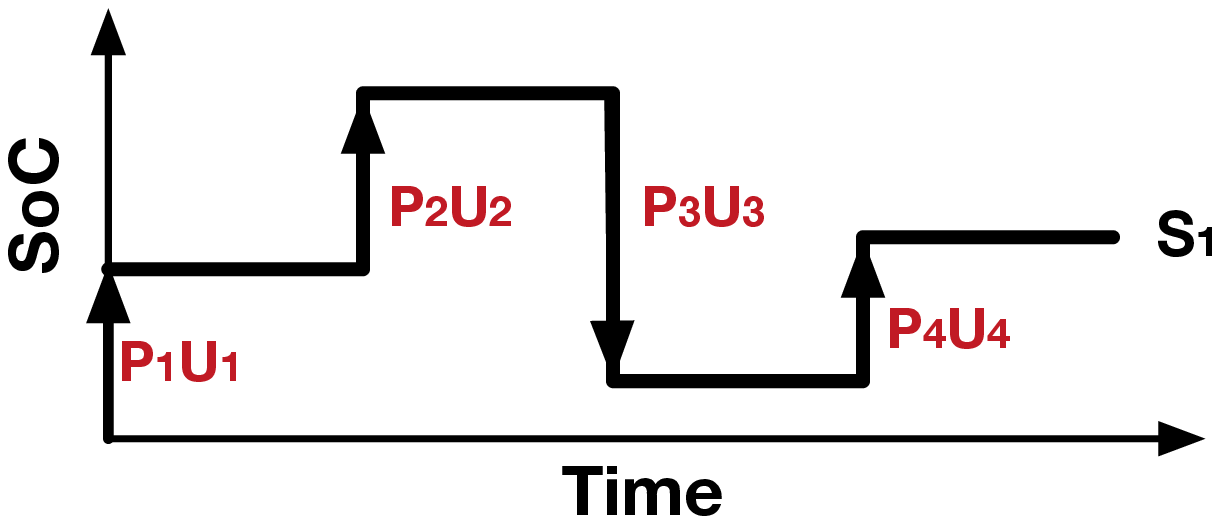}
		\caption{Step function decomposition of SoC. $\mathbf{s_1}$ is decomposed to four step functions $P_1 U_1$, $P_2 U_2$, $P_3 U_3$ and $P_4 U_4$.}
		\label{fig:discreted}
	\end{figure}

\subsection{Initial condition proof}
Since all SoC profiles can be written as the sum of step functions. We first need to prove that $f(\mathbf{s})$ is convex up to one step function as the base case. As shown in Fig. \ref{fig_onechange}, we want to prove that,
	\begin{small}
	\begin{equation*}
	f\left(\lambda \mathbf{s_1} + (1-\lambda)P_i U_i \right) \leq \lambda f(\mathbf{s_1}) + (1-\lambda) f(P_i U_i )\,, \lambda \in [0,1]\,,
	\end{equation*}	
	\end{small}
where $\mathbf{s_1} \in \mathbb{R}^T$, and $P_i U_i$ is a step function with a jump happens at time $i$ with amplitude $P_i$.
\begin{figure}[!ht]
		\centering
		\includegraphics[width=0.8 \columnwidth]{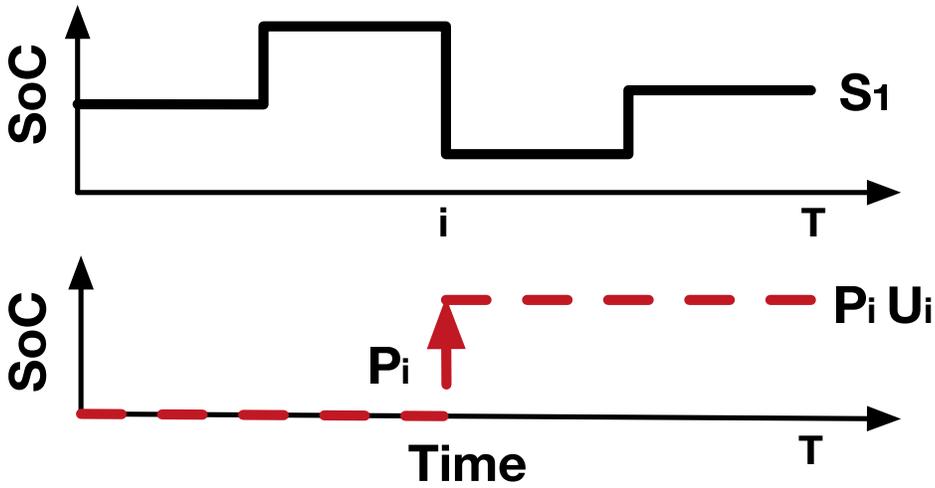}
		\caption{Single step convexity. The upper plot is a SoC profile $\mathbf{s_1} \in \mathbb{R}^T$, the bottom plot is a step function with a amplitude $P_i$ jump at time $i$.}
		\label{fig_onechange}
\end{figure}
	
\subsection{Induction relation}
Having proved the single step convexity, we know that Theorem 1 is true if both $\mathbf{s_1}, \mathbf{s_2} \in \mathbb{R}^{1}$. Now, let's assume that, $f(\mathbf{s})$ is convex up to the sum of $K$ step changes (arranged by time index), 
	$$f\big(\lambda \mathbf{s_1} + (1-\lambda) \mathbf{s_2}\big) \leq \lambda f(\mathbf{s_1}) + (1-\lambda) f(\mathbf{s_2})\,, \lambda \in [0,1]$$
	where $\mathbf{s_1}, \mathbf{s_1} \in \mathbb{R}^{K}$. Considering proof by induction, we need to show $f(\mathbf{s})$ is convex up to the sum of $K+1$ step changes (see Fig. \ref{fig_induction}),
	$$f\big(\lambda \mathbf{s_1} + (1-\lambda) \mathbf{s_2}\big) \leq \lambda f(\mathbf{s_1}) + (1-\lambda) f(\mathbf{s_2})\,, \lambda \in [0,1]$$
	where $\mathbf{s_1}, \mathbf{s_2} \in \mathbb{R}^{K+1}$. 
	
	\begin{figure}[!ht]
		\centering
		\includegraphics[width=0.8 \columnwidth ]{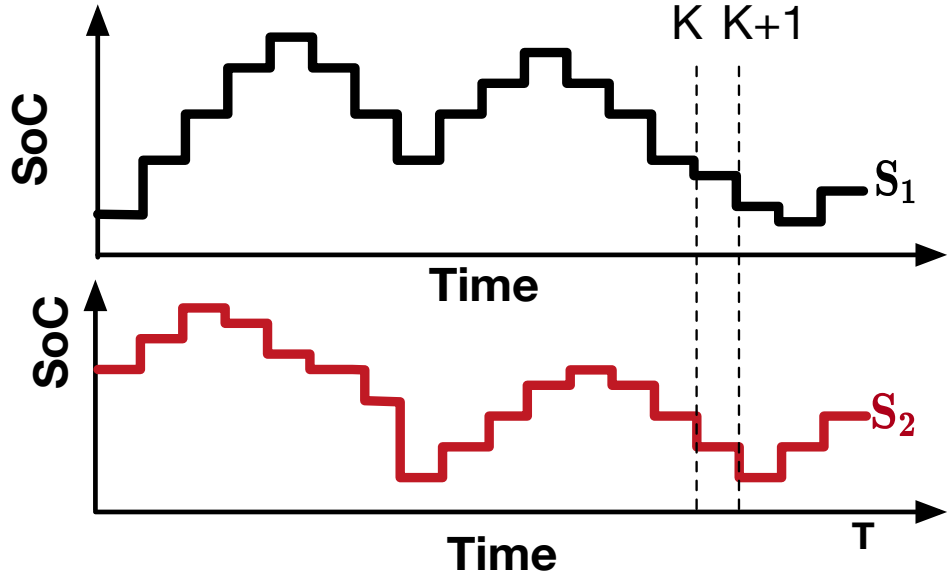}
		\caption{Induction step from K to K+1. Both $\bd S_1$ and $\bd S_2$ are decomposed as the sum of T step functions.}
		\label{fig_induction}
	\end{figure}
The above three-step sketch builds the overall framework for convexity proof. For a more detailed proof step by step, interested readers could refer to Section \ref{sec:proof}.

\section{Subgradient Algorithm}
\label{sec:subgradient}
In this section, we formulate a general optimization problem of battery operation in power market. We use the rainflow cycle-based model for battery degradation calculation. The optimization problem is convex, but the objective function is not differentiable at some points (cycle junction points). Therefore, we provide a subgradient algorithm. With proper step size, the subgradient algorithm is guaranteed to converge to the optimal solution with a user-defined precision level.

\subsection{Battery operation in markets}
Assume we have a battery operating in power markets and providing some profitable service (eg. regulation service, demand response), the battery operation is defined over a finite discrete time intervals, where $t = 1,2,...,T$ with time resolution $t_s$. Following the same formulation in Section \ref{sec:II-A}, we aim to maximize the utility of battery operation over a period of time $t \in [1, T]$.
\begin{subequations}\label{eq:battery_operation}
	\begin{align}
	\max_{\mathbf{c}, \mathbf{d}} U(\cdot) & = R(\mathbf{c}, \mathbf{d}) - \lambda\up{r}\Big[\sum_{i\in N\up{ch}}\Phi(d\up{ch}_i)+\sum_{j\in N\up{dc}}\Phi(d\up{dc}_j)\Big] \,,\label{Eq:PF_obj} \\
	\text{s.t.}  &  [d\up{ch}\;d\up{dc}]^T = \bd{Rainflow} (s(1), \dotsc, s(T))\,,\label{Eq:PF_C1}\\
	& s(1) = s\up{0}\,,\label{Eq:PF_C3}\\
	& s(t+1) = s(t) + \left[c(t)\eta\ud{c}-{d(t)}/{\eta\ud{d}}\right]t_s\,,\label{Eq:PF_C4}\\
	& s\up{min}  \leq s(t) \leq s\up{max}\,,\label{Eq:PF_C5}\\
	& 0\ \leq c(t) \leq P\up{max}\,,\label{Eq:PF_C6}\\
	& 0\ \leq d(t) \leq P\up{max}\,.\label{Eq:PF_C7}
	\end{align}
\end{subequations}
where in Eq. \eqref{Eq:PF_obj}, the first term represents the operation revenue and the second term captures the battery degradation cost by rainflow cycle model. Eq. \eqref{Eq:PF_C3} and \eqref{Eq:PF_C4} describes the evolution of battery SoC, where $s^{0}$ is the given initial state. Battery SoC is limited within $[s\up{min}, s\up{max}]$, in general, $s\up{min}=0$ (empty) and $s\up{max}=1$ (full charge). We include the battery charging/discharging efficiency $\eta_{c}$ and $\eta_{d}$ in the optimization model. The battery output power is subject to power rating in \eqref{Eq:PF_C6}~-~\eqref{Eq:PF_C7}.

Note, different from previous notations, we seperate cycles into charging half cycles (CHC) $d\up{ch}$ and discharging half cycles (DHC) $d\up{dc}$ in Eq. \eqref{Eq:PF_C1} for optimization convenience\footnote{This is because charging and discharging response have different subgradient forms, which will be discussed in Section \ref{sec:IV-II}}. Assume we have $N\up{ch}$ number of CHCs, each one is indexed by $i$ with cycle depth $d\up{ch}_i$; and $N\up{dc}$ number of DHCs, each is indexed by $j$ with depth $d\up{dc}_i$. 

Battery SoC $\mathbf{s}$ is a linear function of battery power output $\mathbf{c}$ and $\mathbf{d}$ according to Eq. \eqref{Eq:PF_C4}. We have proved that the rainflow cycle cost is convex in terms of $\mathbf{s}$. Therefore, the rainflow cycle cost is also convex in terms of $\mathbf{c}$ and $\mathbf{d}$. If the revenue function $R(\cdot)$ is concave, the overall battery operation optimization problem formulated in \eqref{Eq:PF_obj}~-~\eqref{Eq:PF_C7} is a convex optimization problem. 

\subsection{Subgradient algorithm}
\label{sec:IV-II}
To solve the convex battery operation problem in Eq. \eqref{Eq:PF_obj} - \eqref{Eq:PF_C7}, we provide a subgradient algorithm. Using a log-barrier function \cite{boyd2004convex}, we re-write the constrained optimization problem to an unconstrained optimization problem. 
\begin{gather}
\min_{\mathbf{c}, \mathbf{d}} U(\cdot) := -R(\mathbf{c}, \mathbf{d}) + \lambda\up{r}\Big[\sum_{i\in N\up{ch}}\Phi(d\up{ch}_i)+\sum_{j\in N\up{dc}}\Phi(d\up{dc}_j)\Big] \nonumber\\
- \frac{1}{\lambda} \cdot \Big \{\sum_{t=1}^{T} {\log[s^{max}-s(t)]} + \sum_{t=1}^{T} {log[s(t)-s^{min}]} \nonumber\\
+ \sum_{t=1}^{T} {\log[P^{max}-c(t)]} + \sum_{t=1}^{T} {\log[c(t)]} \nonumber\\
+ \sum_{t=1}^{T} {\log[P^{max}-d(t)]} + \sum_{t=1}^{T} {\log[d(t)]} \Big\}
\label{eq:log_barrier_obj}
\end{gather}
when $\lambda \rightarrow +\infty$, the uncontrained problems \eqref{eq:log_barrier_obj} equals to the original constrained problem Eq. \eqref{Eq:PF_obj} - \eqref{Eq:PF_C7}. Note here we change the original maximization problem to an equivalent minimization problem for standard form expression.

The major challenge of solving Eq. \eqref{eq:log_barrier_obj} lies in the second term. We need to find the mathematical relationship between charging cycle depth $d\up{ch}_i$ and charging power $c(t)$, as well as the relationship between discharging cycle depth $d\up{dc}_j$ and discharging power $d(t)$. Recall that the rainflow cycle counting algorithm introcuded in Section \ref{sec:rainflow2}, each time index is mapped to at least one charging half cycle (CHC) or at least one discharging half cycle (DHC). Some time steps sit on the \emph{junction} of two cycles. For example, in Fig. \ref{fig_turning}, $E^{'}$ lies on the junction of discharging cycle $D-G$ and discharging cycle $F-E^{'}$. No time step belongs to more than two cycles. 

Let the time index mapped to the CHC $i$ belong to the set $T\up{ch}_i$, and let the time index mapped to the DHC $j$ belong to the set $T\up{dc}_j$, it follows that
\begin{align}
T\up{ch}_1\cup\dotsc\cup T\up{ch}_{N\up{ch}}\cup T\up{dc}_1\cup\dotsc\cup T\up{dc}_{N\up{ch}} &= T\,,\\
T\up{ch}_i\cap T\up{dc}_j & = \emptyset\,,\, \forall i,j \label{eq:chargingordis}\,.
\end{align}
Eq. \eqref{eq:chargingordis} shows there is no overlapping between a charging and a discharging interval. That is, each half-cycle is either charging or discharging. The cycle depth therefore equals to the sum of battery charging or battery discharging within the cycle time frame, 
\begin{align}
d\up{ch}_i &= \sum_{t\in T\up{ch}_i} \frac{c(t)t_s\eta_{c}}{E}\,,  \label{Eq:rf_ch}\\
d\up{dc}_j &= \sum_{t\in T\up{dc}_j} \frac{d(t)t_s/\eta_{d}}{E}\,, \label{Eq:rf_dc}
\end{align}

The rainflow cycle cost $\sum_{i\in N\up{ch}}\Phi(d\up{ch}_i)$ is not continuously differentiable. At cycle junction points, it has more than one subgradient. Therefore, we use $\partial \sum_{i\in N\up{ch}}\Phi(d\up{ch}_i)|_{c(t)}$ to denote a subgradient at $c(t)$,
\begin{equation}
\partial \sum_{i\in N\up{ch}}\Phi(d\up{ch}_i)|_{c(t)} = \Phi^{'}(d\up{ch}_i) \frac{t_s \cdot \eta_{c}}{E}\,, t\in T\up{ch}_i\,,
\label{Eq:pardf_ch}
\end{equation}
where $d_i^{ch}$ is the depth of cycle that $c(t)$ belongs to. Note, at junction points, $c(t)$ belongs to two cycles so that the subgradient is not unique. We can set $d_i^{ch}$ to any value between $d_{i1}^{ch}$ and $d_{i2}^{ch}$, where $d_{i1}^{ch}$ and $d_{i2}^{ch}$ are the depths of two junction cycles $c(t)$ belongs to. 

Similarly for discharging cycle, a subgradient at $d(t)$ is，
\begin{align}
\partial \sum_{j\in N\up{dc}}\Phi(d\up{dc}_j)|_{d(t)} = \Phi^{'}(d\up{dc}_j) \frac{t_s} {\eta_{d} \cdot E}\,, t\in T\up{dc}_j \label{Eq:pardf_dc}
\end{align}
where $d_j^{dc}$ is the depth of the cycle that $d(t)$ belongs to. At the junction point, $d_j^{dc}$ could be set to any value between $d_{j1}^{dc}$ and $d_{j2}^{dc}$, which are the two junction cycles $d(t)$ belongs to. 

Therefore, we write the subgradient of $U(\cdot)$ with respect to $c(t)$ and $d(t)$ as $\partial U_{c(t)}$ and $\partial U_{d(t)}$, where 
\begin{small}
\begin{align}
\partial U_{c(t)} &= -\frac{\partial R}{\partial c(t)} + \lambda^{r} \Phi^{'}(d\up{ch}_i) \frac{t_s \eta_{c}}{E} \nonumber\\
&- \frac{1}{\lambda} \Big\{\sum_{\tau=t}^{T} \frac{1}{s(\tau)-s^{max}}(\frac{\eta_{c}t_s}{E}) + \sum_{\tau=t}^{T} \frac{1}{s(\tau)-s^{min}}(\frac{\eta_{c}t_s}{E}) \nonumber\\
&+ \frac{1}{c(t)-P^{max}} + \frac{1}{c(t)} \Big\}, t \in T_{i}^{ch}
\end{align}
\end{small}
\begin{small}
\begin{align}
\partial U_{d(t)} &= -\frac{\partial R}{\partial d(t)} + \lambda^{r} \Phi^{'}(d\up{dc}_j) \frac{t_s} {\eta_{d} \cdot E} \nonumber\\
&- \frac{1}{\lambda} \Big\{- \sum_{\tau=t}^{T} \frac{1}{s(\tau)-s^{max}}(\frac{t_s}{\eta_{d} E}) - \sum_{\tau=t}^{T} \frac{1}{s(\tau)-s^{min}}(\frac{t_s}{\eta_{d} E}) \nonumber\\
& + \frac{1}{d(t)-P^{max}} + \frac{1}{d(t)} \Big\}, t \in T_{j}^{dc}
\end{align}
\end{small}

The update rules for $c(t)$ and $d(t)$ at the k-th iteration are,
\begin{equation*}
c_{(k)}(t) = c_{(k-1)}(t) - \alpha_k \cdot \partial U_{c(t)}\,,
\label{eq:charging_update}
\end{equation*}
\begin{equation*}
d_{(k)}(t) = d_{(k-1)}(t) - \alpha_k \cdot \partial U_{d(t)}\,,
\label{eq:discharging_update}
\end{equation*}
where $\alpha_k$ is the step length at k-th iteration. Since the subgradient method is not a decent method \cite{boyd2004}, it is common to keep track of the best point found so far, i.e., the one with smallest function value. At each step, we set
\begin{equation*}
U_{(k)}^{best} = \min \big\{U_{(k-1)}^{best}, U(\mathbf{c}_{(k)}, \mathbf{d}_{(k)})\big\}\,,
\end{equation*}

Since the $U(\cdot)$ is convex, for constant step length ($\alpha_k$ is constant), the subgradient algorithm is guaranteed to converge to the global optima within certain gap bound. 
\begin{equation*}
U_{(k)}^{best}-U^{*} \rightarrow \frac{G^2 \alpha}{2}\,,
\end{equation*}
where $U^{*}$ as the global minimum of Eq. \eqref{eq:log_barrier_obj}, $\alpha$ is the contant step size, and $G$ is the upper bound of the 2-norm of the following sub-gradient vectors,
\begin{align*}
\partial U_{\mathbf{c}} &= [\partial U_{c(1)}, \partial U_{c(2)}, ..., \partial U_{c(T)}]\,, \\
\partial U_{\mathbf{d}} &= [\partial Y_{d(1)}, \partial U_{d(2)}, ..., \partial U_{d(T)}]\,, 
\end{align*}
If we set the step size $\alpha$ to be small enough, we could guarantee the precision of the subgradient algorithm.

\section{Case study}
\label{sec:case}
The proposed battery degradation model has a wide application scope in battery planning and operation problems. In this section, we provide a case study on PJM regulation market to demonstrate the effectiveness of the proposed  degradation model in optimizing BES operation utilities and extending BES lifetime.

\subsection{Frequency regulation market}
We consider the optimal operation of a BES in frequency regulation market. In particular, we adopt a simplified version of the PJM fast regulation market~\cite{bolun_policy}. For market rules, the grid operator pays a per-MW option fee ($\lambda_c$) to a battery storage with stand-by power capacity $C$ that it can provide for the grid. While during the regulation procurement period, a regulation instruction signal $r(t)$ is sent to the battery. Battery should respond to follow the instruction signal and is subjected to a per-MWh regulation mismatch penalty $\lambda^{p}$ for the absolute error between the instructed dispatch Cr(t) and the resource's actual response b(t). Fig. \ref{fig:case_reg} gives an example of the PJM fast regulation signal for 2 hours.

\begin{figure}[!ht]
	\centering
	\includegraphics[width=4in, height = 2in]{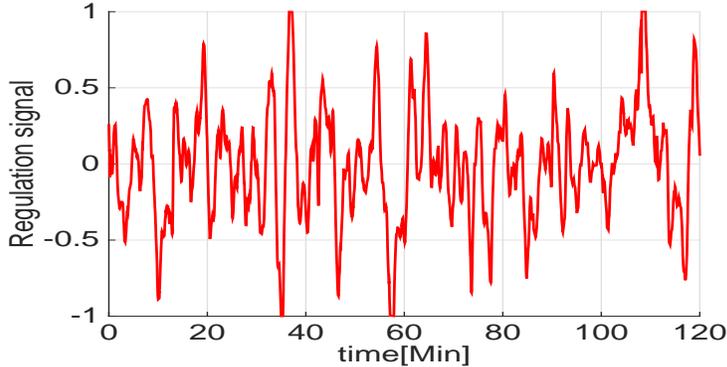}
	\caption{PJM fast regulation signal for 2 hours}
	\label{fig:case_reg}
\end{figure}

We decompose the regulation signal $r(t)$ into charging and discharging parts, $r(t) = r_{d}(t)-r_{c}(t)$,
%
where $r_{c}(t)$ is the charging instruction signal and $r_{d}(t)$ is the discharging signal. Thus, the regulation service revenue is,
\begin{small}
	\begin{equation}
		R(\mathbf{c}, \mathbf{d}) = \lambda_c C \cdot T - \sum_{t\in T} \lambda\up{p}t_s\left|(r_{c}(t)-c(t))+(r_{d}(t)-d(t))\right|\,, \label{eq:reg}
	\end{equation}
\end{small}
where the first part represents the regulation capacity payment and the second part is the mismatch penalty. In this work, we assume the regulation capacity bidding $C$ is fixed, and focus on the optimization of battery response. The regulation revenue $R(\mathbf{c}, \mathbf{d})$ is concave with respect to $\mathbf{c}, \mathbf{d}$. Therefore, plugging \eqref{eq:reg} to the battery market operation model in \eqref{Eq:PF_obj}~-~\eqref{Eq:PF_C7}, we obtain a convex optimization problem for optimizing battery usage in frequency regulation market.

\subsection{Benchmark}
To demonstrate the efficiency of the proposed battery degradation model in maximizing the BES operation utilities as well as extending BES lifetime, we describe two benchmark battery degradation models: assuming zero operating cost \cite{pozo2014} and a linear power-based degradation cost \cite{shi2016leveraging}.

Assume the battery optimization horizon is 2 hours and the time granularity $t_s$ is $4s$, so that $T=1800$. We adopted the regulation market price and linear battery cost model coefficients from paper~\cite{shi2017using}, where the regulation capacity payment is $50\$/MWh$ and the mismatch penalty is $150\$/MWh$. We fixed the regulation capacity bidding as 1MW, with the same value as the battery power capacity. The battery energy capacity $E$ is 15 minute max power output,  and cell replacement price is $0.6 \$/Wh$. Assume battery DoD stress model has a polynomial form as Eq. \eqref{eq:polynomial_dod}, where $k_4 = 4.5e-4$ and $k_5=1.3$~\cite{xu2016modeling}. What's more, we set both the charging/discharging efficiency $\eta_{c}$, $\eta_{d}$ to $0.95$.

\begin{figure}[!ht]
	\centering
	\includegraphics[width=\columnwidth,height=4cm]{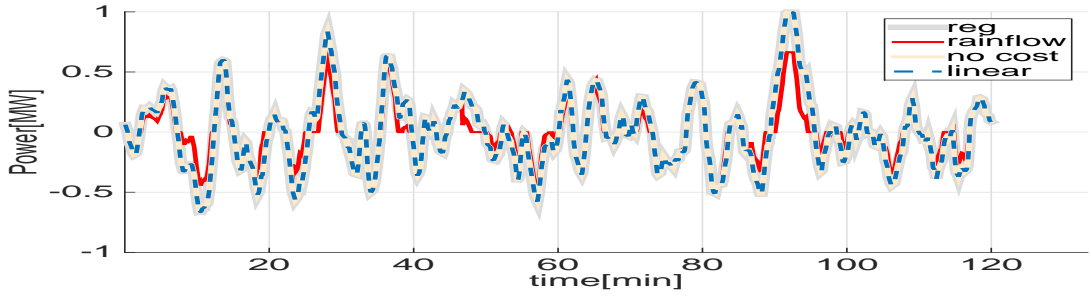}
	\caption{Battery response to the regulation signal under three different battery cost models}
	\label{fig:battery}
\end{figure}
\begin{figure}[!ht]
	\centering
	\includegraphics[width=\columnwidth,height=4cm]{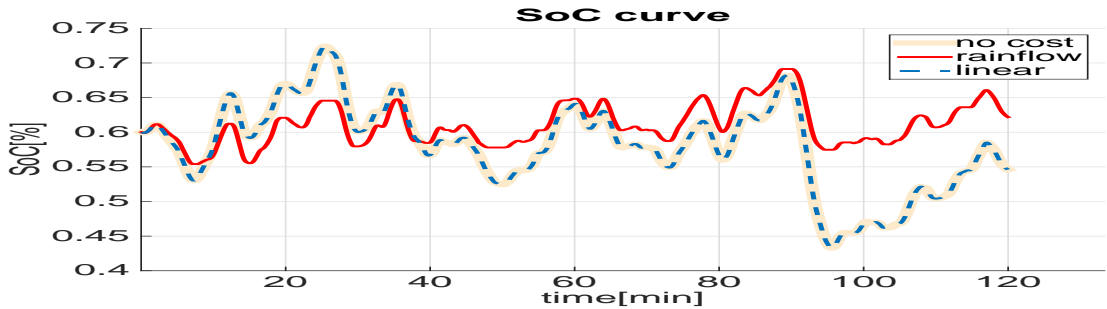}
	\caption{Battery SoC under different battery cost models}
	\label{fig:soc}
\end{figure}
\begin{table}[!t]
	\centering
	\begin{tabular}{|l|l|l|l|}
		\hline
		Annual cost & Rainflow & No cost & Linear\\
		\hline
		\textbf{Total regulation utility (k\$)} & \bf{176} & \bf{137.9} & \bf{137.9}\\
		\hline
		Regulation service payment (k\$) & 338.9 & 438 & 438\\
		\hline
		Modeled battery degradation (k\$) & 162.9 & 0 & 207.2\\
		\hline
		Actual battery degradation (k\$) & 162.9 & 300.1 & 300.1\\
		\hline
		\textbf{Battery life expectancy (month)} & \bf{11.1} & \bf{6} & \bf{6} \\
		\hline
	\end{tabular}
	\caption{Annual economics of BES operation under three battery cost models. ``Modeled battery degradation'' refers to the battery operation cost captured by the model in use, ``actual degradation cost'' is assessed by posterior rainflow algorithm, and ``total regulation utility'' is calculated by payment subtracting actual battery degradation cost.}
	\label{case:T1}
\end{table}

Fig. \ref{fig:battery} and Fig. \ref{fig:soc} compare the power output and SoC evolution for the \emph{same battery} optimized under three different degradation models: assuming no operating cost, linear power-based model and the proposed rainflow cycle-based model. The grey curve in Fig. \ref{fig:battery} is the regulation instruction signal. We find that under no cost model (light yellow curve) and linear power-based degradation model (blue dashed line), the battery response is completely following the regulation instruction signal. However, for the rainflow cycle-based model, instead of  the ``blindly following'' strategy, the battery response strives a good balance between mismatch penalty and degradation cost.
In Fig. \ref{fig:soc}, we observed that the battery SoC is restricted to a moderate range and evolves smoothly for the rainflow cycle-based cost.

Table \ref{case:T1} summarizes the annual economics of the same battery optimized using three different degradation models. We find that, using the rainflow cycle-based degradation model, we have the highest operational utility and longest expected battery lifetime. The total regulation revenue increases by 27.6\% and the battery lifetime almost doubles compared with the other two cases.

\section{Proof}
\label{sec:proof}
Section \ref{sec:convexity} provides a sketch of the convexity proof. Here we extends the sketch and explain the critical proof steps. A full version of the convexity proof could be found in the Appendix.
\begin{proof}
Firstly, we use step function decomposition method to write out $\mathbf{s_1}$, $\mathbf{s_2}$ and $\lambda \mathbf{s_1} + (1-\lambda) \mathbf{s_2}$ as a finite sum of step functions, where
\begin{align}
&\mathbf{s_1} = \sum_{i =1}^{T} Q_i U_i \,, \mathbf{s_2} = \sum_{i =1}^{T} P_i U_i \nonumber \\
&\lambda \mathbf{s_1} + (1-\lambda) \mathbf{s_1}  = \sum_{i =1}^{T} X_i U_i\,, \label{eq:dep_sum_proof}
\end{align}

We obtain the overall proof by induction. We first prove that $f(\mathbf{s})$ is convex up to any \emph{single} step change.

\begin{lemma}
The rainflow cycle-based cost function $f(\mathbf{s})$ is convex up to one step change of the original SoC profile,
\begin{equation*}
f\left(\lambda \mathbf{s_1} + (1-\lambda) P_i U_i\right) \leq \lambda f(\mathbf{s_1}) + (1-\lambda) f(P_i U_i)\,, \lambda \in [0,1]\,,
\label{lemma:single_step_proof}
\end{equation*}
\end{lemma}

We need the following proposition to prove Lemma 1.
\begin{proposition}
	Let $g(\cdot)$ be a convex function where $g(0) = 0$. Let $x_1, x_2, x_3, ..., x_N$ be real numbers, suppose $\sum_{i=1}^{N}x_i = D >0$ and $|x_i|  \leq D, \forall i \in \{1,2,3,...,N\}$
	Then,
	\begin{equation}
	g\left(\sum_{i=1}^{n} x_i\right) \geq \sum_{\{i: x_i \geq 0\}} g(x_i) - \sum_{\{i: x_i < 0\}} g(|x_i|)\,,
	\end{equation}
	\label{l1:deltas}
\end{proposition}

Proposition 1 basically says that the function value of a large number $D$ is larger than the sum of function values of N small numbers $x_1,...,x_N$, where $\sum_{i=1}^{N}x_i = D$.  Now, let's consider a step function added to $\lambda \mathbf{s_1}$, where $\mathbf{{s_1}^{'}} = \lambda \mathbf{s_1} + (1-\lambda) P_i U_i$. Suppose the rainflow cycle counting results are,
\begin{align*}
\bd{Rainflow} (\lambda \mathbf{s_1}): & \underbrace{\lambda d_1, \lambda d_2, ..., \lambda d_M, 0, 0,...}_{L}\,,\\
\bd{Rainflow} (\mathbf{s_1}^{'}): & \underbrace{{d_1}^{'}, {d_2}^{'}, ..., {d_N}^{'}, 0, 0,...}_{L}\,,
\end{align*}
where $[d_1, d_2, ..., d_M]$ are cycle counting results for $\mathbf{s_1}$. We add some cycles with 0 depth in the end to ensure that $\bd{Rainflow}(\lambda \mathbf{s_1})$ and $\bd{Rainflow}(\mathbf{s_1}^{'})$ have the same number of cycles. Define $\Delta d_i$ such that,
\begin{equation}
d_i^{'} = \lambda d_i + (1-\lambda) \Delta d_i, \forall i = 1,2,...,L\,,
\end{equation}
We observe that,
\begin{align*}
\left|\sum_{i=1}^{L} \Delta d_i \right|  \leq \left| P_i \right| \mbox{ and }
\left|\Delta d_i\right|  \leq \left|P_i\right|\,,
\end{align*}

When $\sum_{i=1}^{L} \Delta d_i = \left|P_i\right|$, Lemma \ref{lemma:single_step_proof} is directly proved by applying Proposition 1. When $-\left|P_i\right| \leq \sum_{i=1}^{L} \Delta d_i < |P_i|$, we first add some ``virtual cycles'' to make $\sum_{i=1}^{L} \Delta d_i = |P_i|$. Then we prove the original cost $f\big(\lambda \mathbf{s_1} + (1-\lambda)U_i P_i\big)$ is strictly less than the cost after adding the ``virtual cycles''.

%
%

Lemma \ref{lemma:single_step_proof} proved the convexity for base case. Next, we want to prove the induction relation from $K$ to $K+1$ step. Assume that $f(\mathbf{s})$ is convex up to the sum of $K$ step changes (arranged by time index), let's show $f(\mathbf{s})$ is convex up to the sum of $K+1$ step changes, where
\begin{small}
\begin{equation*}
f\left(\lambda \mathbf{s_1} + (1-\lambda) \mathbf{s_1}\right) \leq \lambda f(\mathbf{s_1}) + (1-\lambda) f(\mathbf{s_1})\,, \lambda \in [0,1]\,,
\end{equation*}
\end{small}
where $\mathbf{s_1}, \mathbf{s_1} \in \mathbb{R}^{K+1}$. The following Proposition is needed for the proof.
\begin{proposition}\label{propo:comb1}
	\begin{small}
		\begin{equation*}
		f\left(\sum_{t=1}^{K}X_t U_t\right) \!\geq\!f\left(\sum_{t=1}^{i\!-\!1} X_t U_t \!+\! (X_i\!+\!X_{i+1})U_i\!+\!\sum_{t=i\!+\!2}^{K}X_t U_t\right)\,,
	\end{equation*}
	\end{small}
	In other words, the cycle stress cost will reduce if combining adjacent unit changes.
\end{proposition}

Recall the step function decomposition results for $\mathbf{s_1}$, $\mathbf{s_2}$ and $\lambda \mathbf{s_1} + (1-\lambda) \mathbf{s_2}$ in (\ref{eq:dep_sum_proof}). There are three cases when we go from $T=K$ to $T=K+1$, classified by the value and sign of $X_{K}$, $X_{K+1}$.

\noindent \textbf{Case 1:} $X_{K}$ and $X_{K+1}$ are same direction.

If $X_{K+1}$ and $X_{K}$ are same direction, we could move $X_{K+1}$ to the previous step without affecting the total cost $f(\mathbf{s_1} + (1-\lambda) \mathbf{s_2})$. Then we prove the $K+1$ convexity by applying Proposition \ref{propo:comb1}.
\begin{align*}
		& f\left(\lambda \mathbf{s_1} + (1-\lambda) \mathbf{s_2}\right) \nonumber\\
		= & f\left(\lambda \mathbf{s_1}^{K} + (1-\lambda) \mathbf{s_2}^{K} + X_{K+1}U_{K}\right) \nonumber\\
		= & f\left\{\lambda \mathbf{s_1}^{K} + (1-\lambda) \mathbf{s_2}^{K} + [\lambda Q_{K+1} + (1-\lambda) P_{K+1}] U_{K}\right\}\nonumber\\
		\leq & \lambda f\left(\mathbf{s_1}^{K}+ Q_{K+1}U_{K}\right) + (1-\lambda)f\left(\mathbf{s_2}^{K}+ P_{K+1}U_{K}\right) \nonumber\\
		\leq & \lambda f(\mathbf{s_1}) + (1-\lambda)f(\mathbf{s_2}) \ (\text{by Lemma \ref{lemma:single_step_proof}})
\end{align*}

\noindent \textbf{Case 2:} $X_{K}$ and $X_{K+1}$ are different directions, with $|X_{K}| \geq |X_{K+1}|$. In this case, the last step $X_{K+1}$ could be separated out from the previous SoC profile. Therefore
\begin{small}
\begin{align*}
	& f\left(\lambda \mathbf{s_1} + (1-\lambda) \mathbf{s_2}\right) & \nonumber\\
	= & f\left(\lambda \mathbf{s_1}^{K} + (1-\lambda) \mathbf{s_2}^{K}\right) + \Phi\left(X_{K+1} U_{K+1}\right) \nonumber\\
	\leq & \lambda f(\mathbf{s_1}^{K}\!)\!+\!(1\!-\!\lambda)f(\mathbf{s_2}^{K}\!)\!+\! \Phi\left[\lambda Q_{K\!+\!1} U_{K\!+\!1}\!+\!(1\!-\!\lambda) P_{K\!+\!1}U_{K\!+\!1}\right] \nonumber\\
	\leq & \lambda \big[f(\mathbf{s_1}^{K}\!) \!+\! \Phi(Q_{K\!+\!1}U_{K+1}\!)\big]\!+\! (1\!-\!\lambda)\big[f(\mathbf{s_2}^{K}\!) \!+\! \Phi(P_{K\!+\!1} U_{K+1})\big] & \nonumber\\
	\leq &\lambda f(\mathbf{s_1}) + (1-\lambda)f(\mathbf{s_2}) &
\end{align*}
\end{small}
\noindent \textbf{Case 3:} $X_{K}$ and $X_{K+1}$ are different directions, with $|X_{K}| < |X_{K+1}|$. In such condition, $X_{K+1}$ is not easily separated out from previous SoC. It further contains three cases.
\begin{itemize}
\item $X_{K-1}$ and $X_{K}$ are the same direction. We could use the same ``trick" in Case 1 to combine step $K-1$ and $K$. Proof is trivial for this case.

\item $X_{K-1}$ and $X_{K}$ are different directions, $X_{K}$ and $X_{K+1}$ form a cycle that is \emph{separate} from the rest of the signal (eg. it is the deepest cycle). We can separate $X_{K+1}$ out, and proof will be similar to Case 2.

\item $X_{K-1}$ and $X_{K}$ are different directions, $X_{K}$ and $X_{K+1}$ do not form a separate cycle. This condition is the most complicated case, since it's hard to move $X_{K+1}$ to the previous step, or separate it out. Therefore, we need to look into $Q_K$, $Q_{K+1}$, $P_K$, $P_{K+1}$ in order to show the $K+1$ step convexity. It further contains four sub-cases as given in Fig. \ref{Fig:c4}, showing convexity for each sub-case finishes the overall convexity proof.
\begin{figure}[!ht]
	\centering
	\includegraphics[width= 0.6\columnwidth]{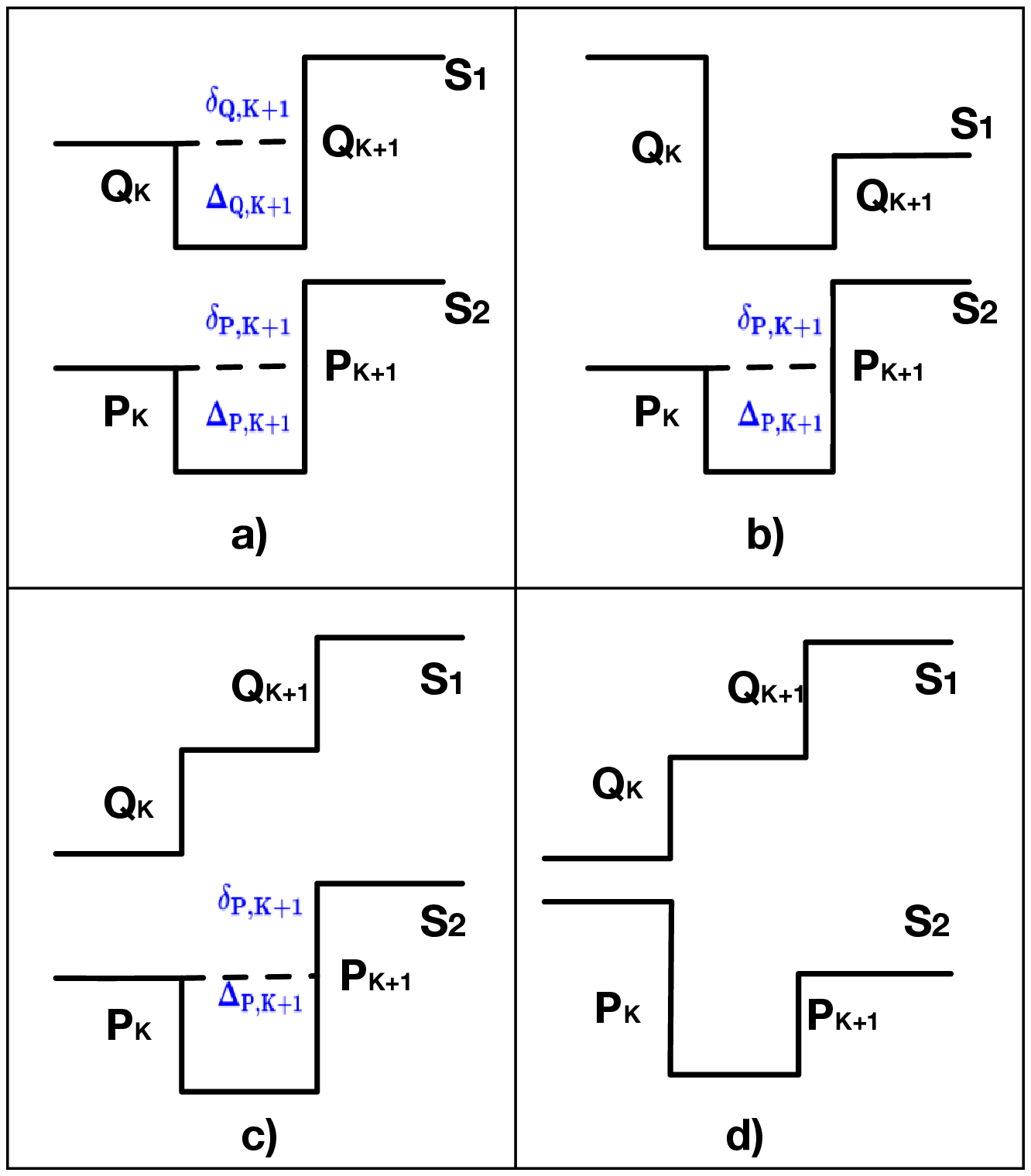}
	\caption{Four cases of $Q_K$, $Q_{K+1}$, $P_K$, $P_{P+1}$}
	\label{Fig:c4}
\end{figure}
\end{itemize}
\end{proof}

\section{Conclusion}
\label{sec:con}
Battery operational cost modeling is very important for BES planning and operation.
An ideal cost model needs to both capture the fundamental battery degradation, and be easy to incorporate in optimization and computational tractable to solve. Cycle-based degradation models, mainly using the rainflow algorithm for cycle counting is electrochemically accurate. However, it does not have a closed form thus hard to be optimized. In this paper, we prove the rainflow cycle-based degradation cost is convex (with respect to charging/discharging power). Convexity enables the degradation model easy to be incorporated and guarantee the solution quality. We provide a subgradient solver algorithm. The proposed degradation model and subgradient algorithm have a broad application scope in various battery planning and operation optimization problems. From the case study in frequency regulation market, we verified the proposed degradation model can significantly improve the operational utility and extend battery lifetime. In future, we will apply the proposed degradation model to more BES applications besides frequency regulation.

\bibliographystyle{IEEEtran}
\bibliography{dc_bib}

\begin{thebibliography}{10}
\providecommand{\url}[1]{#1}
\csname url@samestyle\endcsname
\providecommand{\newblock}{\relax}
\providecommand{\bibinfo}[2]{#2}
\providecommand{\BIBentrySTDinterwordspacing}{\spaceskip=0pt\relax}
\providecommand{\BIBentryALTinterwordstretchfactor}{4}
\providecommand{\BIBentryALTinterwordspacing}{\spaceskip=\fontdimen2\font plus
\BIBentryALTinterwordstretchfactor\fontdimen3\font minus
  \fontdimen4\font\relax}
\providecommand{\BIBforeignlanguage}[2]{{%
\expandafter\ifx\csname l@#1\endcsname\relax
\typeout{** WARNING: IEEEtran.bst: No hyphenation pattern has been}%
\typeout{** loaded for the language `#1'. Using the pattern for}%
\typeout{** the default language instead.}%
\else
\language=\csname l@#1\endcsname
\fi
#2}}
\providecommand{\BIBdecl}{\relax}
\BIBdecl

\bibitem{dunn2011electrical}
B.~Dunn, H.~Kamath, and J.-M. Tarascon, ``Electrical energy storage for the
  grid: a battery of choices,'' \emph{Science}, vol. 334, no. 6058, pp.
  928--935, 2011.

\bibitem{shi2016leveraging}
Y.~Shi, B.~Xu, B.~Zhang, and D.~Wang, ``Leveraging energy storage to optimize
  data center electricity cost in emerging power markets,'' in
  \emph{Proceedings of the Seventh International Conference on Future Energy
  Systems}.\hskip 1em plus 0.5em minus 0.4em\relax ACM, 2016, p.~18.

\bibitem{li2011optimal}
N.~Li, L.~Chen, and S.~H. Low, ``Optimal demand response based on utility
  maximization in power networks,'' in \emph{Power and Energy Society General
  Meeting, 2011 IEEE}.\hskip 1em plus 0.5em minus 0.4em\relax IEEE, 2011, pp.
  1--8.

\bibitem{li2016optimal}
P.~Li and B.~Zhang, ``An optimal treatment assignment strategy to evaluate
  demand response effect,'' \emph{arXiv preprint arXiv:1610.00362}, 2016.

\bibitem{bitar2011role}
E.~Bitar, R.~Rajagopal, P.~Khargonekar, and K.~Poolla, ``The role of co-located
  storage for wind power producers in conventional electricity markets,'' in
  \emph{American Control Conference (ACC), 2011}.\hskip 1em plus 0.5em minus
  0.4em\relax IEEE, 2011, pp. 3886--3891.

\bibitem{pandzic2015near}
H.~Pandžić, Y.~Wang, T.~Qiu, Y.~Dvorkin, and D.~S. Kirschen, ``Near-optimal
  method for siting and sizing of distributed storage in a transmission
  network,'' \emph{IEEE Transactions on Power Systems}, vol.~30, no.~5, Sept
  2015.

\bibitem{zakeri2015electrical}
B.~Zakeri and S.~Syri, ``Electrical energy storage systems: A comparative life
  cycle cost analysis,'' \emph{Renewable and Sustainable Energy Reviews},
  vol.~42, pp. 569--596, 2015.

\bibitem{vetter2005ageing}
J.~Vetter, P.~Nov{\'a}k, M.~Wagner, C.~Veit, K.-C. M{\"o}ller, J.~Besenhard,
  M.~Winter, M.~Wohlfahrt-Mehrens, C.~Vogler, and A.~Hammouche, ``Ageing
  mechanisms in lithium-ion batteries,'' \emph{Journal of power sources}, vol.
  147, no.~1, pp. 269--281, 2005.

\bibitem{nykvist2015rapidly}
B.~Nykvist and M.~Nilsson, ``Rapidly falling costs of battery packs for
  electric vehicles,'' \emph{Nature Climate Change}, vol.~5, no.~4, pp.
  329--332, 2015.

\bibitem{guo2013electricity}
Y.~Guo and Y.~Fang, ``Electricity cost saving strategy in data centers by using
  energy storage,'' \emph{IEEE Transactions on Parallel and Distributed
  Systems}, vol.~24, no.~6, pp. 1149--1160, 2013.

\bibitem{ecker2014calendar}
M.~Ecker, N.~Nieto, S.~K{\"a}bitz, J.~Schmalstieg, H.~Blanke, A.~Warnecke, and
  D.~U. Sauer, ``Calendar and cycle life study of li (nimnco) o 2-based 18650
  lithium-ion batteries,'' \emph{Journal of Power Sources}, vol. 248, pp.
  839--851, 2014.

\bibitem{xu2016modeling}
B.~Xu, A.~Oudalov, A.~Ulbig, G.~Andersson, and D.~Kirschen, ``Modeling of
  lithium-ion battery degradation for cell life assessment,'' \emph{IEEE
  Transactions on Smart Grid}, 2016.

\bibitem{rychlik1987new}
I.~Rychlik, ``A new definition of the rainflow cycle counting method,''
  \emph{International journal of fatigue}, vol.~9, no.~2, pp. 119--121, 1987.

\bibitem{downing1982simple}
S.~D. Downing and D.~Socie, ``Simple rainflow counting algorithms,''
  \emph{International Journal of Fatigue}, vol.~4, no.~1, pp. 31--40, 1982.

\bibitem{amzallag1994standardization}
C.~Amzallag, J.~Gerey, J.~Robert, and J.~Bahuaud, ``Standardization of the
  rainflow counting method for fatigue analysis,'' \emph{International journal
  of fatigue}, vol.~16, no.~4, pp. 287--293, 1994.

\bibitem{muenzel2015multi}
V.~Muenzel, J.~de~Hoog, M.~Brazil, A.~Vishwanath, and S.~Kalyanaraman, ``A
  multi-factor battery cycle life prediction methodology for optimal battery
  management,'' in \emph{Proceedings of the 2015 ACM Sixth International
  Conference on Future Energy Systems}.\hskip 1em plus 0.5em minus 0.4em\relax
  ACM, 2015, pp. 57--66.

\bibitem{dragivcevic2014capacity}
T.~Dragi{\v{c}}evi{\'c}, H.~Pand{\v{z}}i{\'c}, D.~{\v{S}}krlec, I.~Kuzle, J.~M.
  Guerrero, and D.~S. Kirschen, ``Capacity optimization of renewable energy
  sources and battery storage in an autonomous telecommunication facility,''
  \emph{IEEE Transactions on Sustainable Energy}, vol.~5, no.~4, pp.
  1367--1378, 2014.

\bibitem{musallam2012efficient}
M.~Musallam and C.~M. Johnson, ``An efficient implementation of the rainflow
  counting algorithm for life consumption estimation,'' \emph{IEEE Transactions
  on Reliability}, vol.~61, no.~4, pp. 978--986, 2012.

\bibitem{he2016optimal}
G.~He, Q.~Chen, C.~Kang, P.~Pinson, and Q.~Xia, ``Optimal bidding strategy of
  battery storage in power markets considering performance-based regulation and
  battery cycle life,'' \emph{IEEE Transactions on Smart Grid}, vol.~7, no.~5,
  pp. 2359--2367, 2016.

\bibitem{abdulla2016optimal}
K.~Abdulla, J.~De~Hoog, V.~Muenzel, F.~Suits, K.~Steer, A.~Wirth, and
  S.~Halgamuge, ``Optimal operation of energy storage systems considering
  forecasts and battery degradation,'' \emph{IEEE Transactions on Smart Grid},
  2016.

\bibitem{koller2013defining}
M.~Koller, T.~Borsche, A.~Ulbig, and G.~Andersson, ``Defining a degradation
  cost function for optimal control of a battery energy storage system,'' in
  \emph{2013 IEEE Grenoble Conference}, June 2013, pp. 1--6.

\bibitem{tran2013energy}
D.~Tran and A.~M. Khambadkone, ``Energy management for lifetime extension of
  energy storage system in micro-grid applications,'' \emph{IEEE Transactions
  on Smart Grid}, vol.~4, no.~3, Sept 2013.

\bibitem{millner2010modeling}
A.~Millner, ``Modeling lithium ion battery degradation in electric vehicles,''
  in \emph{Innovative Technologies for an Efficient and Reliable Electricity
  Supply (CITRES), 2010 IEEE Conference on}.\hskip 1em plus 0.5em minus
  0.4em\relax IEEE, 2010, pp. 349--356.

\bibitem{boyd2004convex}
S.~Boyd and L.~Vandenberghe, \emph{Convex optimization}.\hskip 1em plus 0.5em
  minus 0.4em\relax Cambridge university press, 2004.

\bibitem{boyd2004}
B.~Stephen, L.~Xiao, and A.~Mutapcic, ``Subgradient methods,'' in \emph{lecture
  notes of EE392, Stanford University}, Autume 2004.

\bibitem{bolun_policy}
B.~Xu, Y.~Dvorkin, D.~S. Kirschen, C.~A. Silva-Monroy, and J.~P. Watson, ``A
  comparison of policies on the participation of storage in u.s. frequency
  regulation markets,'' in \emph{2016 IEEE Power and Energy Society General
  Meeting (PESGM)}, July 2016.

\bibitem{pozo2014}
D.~Pozo, J.~Contreras, and E.~E. Sauma, ``Unit commitment with ideal and
  generic energy storage units,'' \emph{IEEE Transactions on Power Systems},
  vol.~29, no.~6, Nov 2014.

\bibitem{shi2017using}
Y.~Shi, B.~Xu, D.~Wang, and B.~Zhang, ``Using battery storage for peak shaving
  and frequency regulation: Joint optimization for superlinear gains,''
  \emph{arXiv preprint arXiv:1702.08065}, 2017.

\end{thebibliography}

\newpage
\section{Appendix}
\label{sec:appendix}
\setcounter{lemma}{0}
\setcounter{theorem}{0}
\setcounter{proposition}{0}
Here we provide a detailed proof that the rainflow cycle life loss model is convex. We use $\bd s$ to denote the battery SoC profile, and the battery degradation cost based on the rainflow method is denoted by f, which is defined by,
\begin{equation}
f(\bd s) = \sum_{i =1}^{N} \Phi(d_i)\,,
\label{equ:cost}
\end{equation}
where $\Phi$ is a convex depth of discharge (DoD) stress function, $N$ is the number charging/discharging cycles, and $d_i$ is the depth of the ith charging/discharging cycle. 

\begin{theorem}{Cost Model Convexity}
	
	The rainflow cycle-based battery life loss model $f(\bd s) = \sum_{i=1}^{N} \Phi(d_i)$, where $[d_1, d_2,...,d_{N}] = \bd {Rainflow} (\bd s)$ is convex given function $\Phi(\cdot)$ is convex. In mathematical form, $\forall \bd s_1, \bd s_2 \in \mathbb{R}^{T}$:
	\begin{small}
		\begin{equation}
		f(\lambda \bd s_1 + (1-\lambda) \bd s_2) \leq \lambda f(\bd s_1) + (1-\lambda) f(\bd s_2), \forall \lambda \in [0,1]
		\end{equation}
	\end{small}
	\label{thm_append}
\end{theorem}

\subsection{Step function decomposition}
	First, we introduce the step function decompostion of SOC signal.
	Notice that any SoC series $\bd S$, could be written out as a finite sum of step functions:
	\begin{align}
	\bd s = \sum_{i =1}^{T} P_i u(t-i)\,, \label{eq:dep_s1}
	\end{align}
	where u(t-i) is the unit step function and $P_i$ is the signal amplitude.
	\begin{align}
	u(t-i) = 
	\begin{cases}
	1& \text{t $\geq$ i}\\
	0 & \text{otherwise}
	\end{cases}, \ \ \forall t = 1,2,...,T
	\end{align}
	
	For notation convenience, we use $U_i$ to denote $u(t-i)$. Fig. \ref{fig_append_decom} gives an example of the step function decomposition method. An SoC signal $s_1$ is  decomposed into 4 step functions $P_1U_1$, $P_2U_2$, $P_3U_3$ and $P_4U_4$.
\begin{figure}[H]
	\centering
	\includegraphics[width= 0.7 \columnwidth]{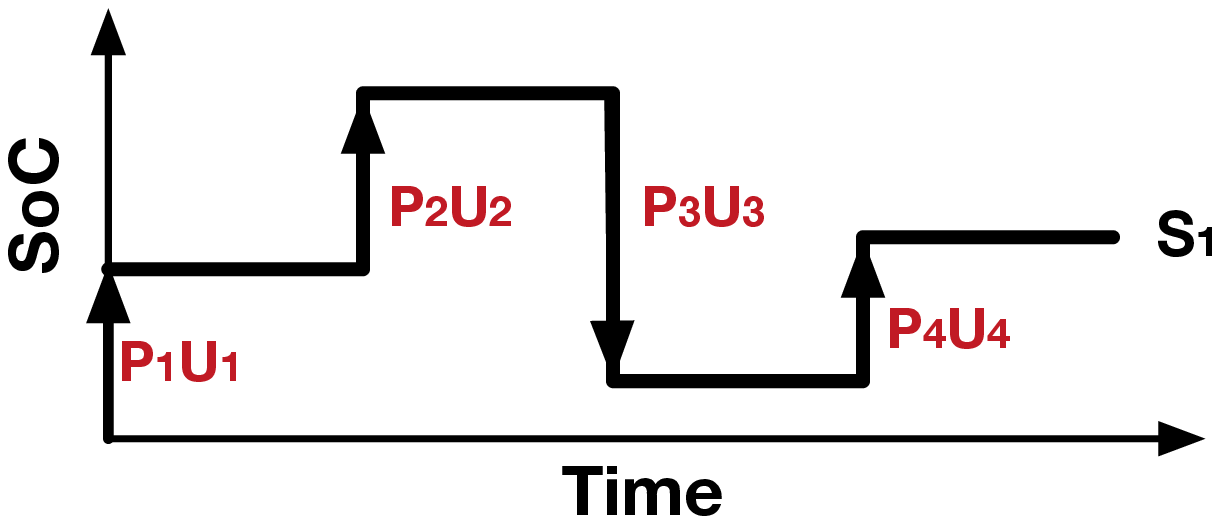}
	\caption{Step function decomposition of $S_1$}
	\label{fig_append_decom}
\end{figure}

\subsection{Single step change convexity}
Since all SoC profile can be written as the sum of step functions, by induction method, we first need to prove that $f(\bd s)$ is convex up to a step function as base case. 

\begin{lemma}{Rainflow Convexity for single step}
	
	The rainflow cycle-based degradation cost function is convex up to one step change of the orginal SoC profile, where
	$$f\big(\lambda \bd s_1 + (1-\lambda)P_i U_i\big) \leq \lambda f(\bd s_1) + (1-\lambda) f(P_i U_i)\,, \lambda \in [0,1]$$
	\label{append_lemma1}
\end{lemma}

In order to prove the Lemma \ref{append_lemma1}, we need the following propositions.
\begin{proposition}
	Let $g(\cdot)$ be a convex function where $g(0) = 0$. Let $x_1$, $x_2$ be positive real numbers. Then
	$$g(x_1+x_2) \geq g(x_1)+g(x_2)\,,$$
	\label{append_c1}
\end{proposition}
\begin{proof}
	By convexity of $g$, we have 
	$$\frac{x_1}{x_1+x_2}g(x_1+x_2) + \frac{x_2}{x_1+x_2}g(0) \geq g(x_1)\,,$$
	and
	$$\frac{x_2}{x_1+x_2}g(x_1+x_2) + \frac{x_1}{x_1+x_2}g(0) \geq g(x_2)\,,$$
	Adding the two equations finish the proof.
\end{proof}

\begin{proposition}
	Let $g(\cdot)$ be a convex function where $g(0) = 0$. Let $x_1$, $x_2$ be positive real numbers, and $x_1 \geq x_2$. Then
	$$g(x_1-x_2) \leq g(x_1)-g(x_2)\,, $$
	\label{append_c2}
\end{proposition}
\begin{proof}
	By Proposition \ref{append_c1}, $$g(x+y) \geq g(x)+g(y), \forall x, y>0\,,$$
	Let $x = x_1-x_2>0$, $y = x_2 >0$, so that
	$$g(x_1-x_2+x_2)  \geq g(x_1-x_2)+g(x_2)\,,$$
	Q.E.D.
\end{proof}
\begin{proposition}
	Let $g(\cdot)$ be a convex function where $g(0) = 0$. Let $x_1 \geq x_2 >0$ be positive real numbers. Then
	$$g(\frac{1}{2} x_1- \frac{1}{2} x_2) \leq \frac{1}{2} g(x_1) - \frac{1}{2} g(x_2)\,,$$
\end{proposition}
\begin{proof}
	From Proposition \ref{append_c2}, 
	$$g(\frac{1}{2} x_1- \frac{1}{2}x_2) \leq g(\frac{1}{2} x_1) - g(\frac{1}{2}x_2), \forall x_1 \geq x_2 >0$$
	Therefore, it suffices to show, 
	$$g(\frac{1}{2} x_1) - g(\frac{1}{2} x_2) \leq \frac{1}{2} g(x_1) - \frac{1}{2} g(x_2)\,,$$
	Define $h(z) = g(\frac{1}{2} z)- \frac{1}{2} g(z)$, 
	$$h'(z) = \frac{1}{2} g^{'}\big(\frac{1}{2} z\big) - \frac{1}{2} g\big(z\big) = \frac{1}{2} [g^{'}\big(\frac{1}{2} z\big)-g^{'}\big(z\big)]< 0\,,$$
	$h(\cdot)$ is a monotone decreasing function. For $x_1 \geq x_2>0$,
	\begin{align*}
	h(x_1) & \leq h(x_2)\\
	g(\frac{1}{2} x_1)- \frac{1}{2} g(x_1) & \leq g(\frac{1}{2} x_2)- \frac{1}{2} g(x_2)\,,\\
	g(\frac{1}{2} x_1) - g(\frac{1}{2} x_2) & \leq \frac{1}{2} g(x_1) - \frac{1}{2} g(x_2)
	\end{align*}
	
	If $g(\cdot)$ is continous, we can generalize the midpoint property to a more broad $\lambda$,
	$$g(\lambda x_1- (1-\lambda) x_2) \leq \lambda g(x_1) - (1-\lambda) g(x_2), \forall \lambda x_1\geq  (1-\lambda) x_2 > 0$$ 
	Q.E.D.
\end{proof}

\begin{proposition}
	Let $g(\cdot)$ be a convex function where $g(0) = 0$. Let $x_1, x_2, x_3$ be positive real numbers, which satisfy that $x_1+x_2-x_3 \geq 0$, and $x_i \leq x_1+x_2-x_3, \forall i \in [1,2,3]$. Then
	$$g(x_1+x_2-x_3) \geq g(x_1)+g(x_2) - g(x_3)\,,$$
	\label{append_c4}
\end{proposition}
\begin{proof}
	From $x_1 \leq x_1+x_2-x_3$ we have $x_2 \geq x_3$. From $x_2 \leq x_1+x_2-x_3$, we have $x_1 \geq x_3$
	
	Let's further assume $x_1 \geq x_2$,
	$$g(x_1+x_2-x_3)-g(x_1) = (x_2-x_3) \cdot g^{'}(\theta_1), \theta_1 \in [x_1,x_1+x_2-x_3]\,,$$
	$$g(x_2)-g(x_3) = (x_2-x_3) \cdot g^{'}(\theta_2), \theta_2 \in [x_3,x_2]\,,$$
	Since $g(\cdot)$ is a convex function, for $\theta_2 \leq x_2 \leq x_1 \leq \theta_1$, we have $g^{'}(\theta_2) \leq g^{'}(\theta_1)$.
	Therefore, 
	$$g(x_2)-g(x_3) \leq g(x_1+x_2-x_3)-g(x_1)\,,$$
	$$g(x_1+x_2-x_3) \geq g(x_1)+g(x_2)-g(x_3)$$
	
	If $x_1 < x_2$, similarly we have
	$$g(x_1)-g(x_3) \leq g(x_1+x_2-x_3)-g(x_2)\,,$$
	$$g(x_1+x_2-x_3) \geq g(x_1)+g(x_2)-g(x_3)$$
	Q.E.D
\end{proof}

\begin{proposition}
	Let $g(\cdot)$ be a convex function where $g(0) = 0$. Let $x_1, x_2, x_3, ..., x_n$ be real numbers, suppose 
	\begin{itemize}
		\item $\sum_{i=1}^{n}x_i = D >0$
		\item $|x_i|  \leq D, \forall i \in \{1,2,3,...,n\}$
	\end{itemize}
	Then,
	$$g(\sum_{i=1}^{n} x_i) \geq \sum_{\{i: x_i \geq 0\}} g(x_i)  - \sum_{\{i: x_i < 0\}} g(|x_i|)\,,$$
	\label{append_c5}
\end{proposition}

\begin{proof}
	(1) If all $x_n$'s are positive, it is trivial to show $g(\sum_{i=1}^{n} x_i) \geq \sum_{i} g(x_i)$ by Proposition \ref{append_c1}.
	
	(2) If $x_n$ contains both positive and negative numbers, we order them in an ascending order and renumber them as,
	
	$$x_1 \leq x_2 \leq ... \leq 0 \leq ... \leq x_{n}\,,$$ 
	
	Pick $x_1$ (the most negative number), and find some positive $x_i$, $x_{i+1}$ such that 
	$$x_{i+1} \geq x_{i} \geq |x_1| >0\,,$$
	
	Applying Proposition \ref{append_c4}, we have 
	$$g(x_{i+1}+x_{i}+x_{1}) = g(x_{i+1} + x_{i} - |x_1|) \geq g(x_{i+1}) +g (x_{i}) - g(|x_1|)\,,$$
	
	Note, if we can not find such $x_{i}$, $x_{i+1}$, eg. $x_n < |x_1|$. We could group a bunch of \emph{postive} $x_i$'s to form two new variables $y_1 = \sum_{i \in N_1} x_i$, $y_2 = \sum_{i \in N_2} x_i$ where $N_1 \cap N_2 = \emptyset$.  For sure there exists such $y_1 \geq y_2 \geq |x_1|$, since
	\begin{align*}
	|\sum_{i=1}^{n} x_i| & = |\sum_{i: x_i \geq 0} x_i + \sum_{j: x_j < 0, j \neq 1} x_j + x_1| \\
	& =  \sum_{i: x_i \geq 0} x_i  - |\sum_{j: x_j < 0, j \neq 1} x_j| - |x_1|\\
	& = D
	\end{align*}	
	$$\sum_{i: x_i \geq 0} x_i = |\sum_{j: x_j < 0, j \neq 1} x_j| + |x_1| + D \geq 2|x_1|$$
	
	Applying Proposition \ref{append_c4},
	\begin{align*}
	g(y_{1}+y_{2}+x_{1}) & = g(y_{1} + y_{2} - |x_1|)\\
	& \geq g(y_{1}) +g (y_{2}) - g(|x_1|)\\
	& \geq \sum_{i \in N_1}  g(x_i) + \sum_{j \in N_2} g(x_j) - g(|x_1|), N_1 \cap N_2 = \emptyset
	\end{align*}
	
	Define $x_1^{'} = x_{i+1}+x_{i}+x_{1} > 0$ and re-order $x_1^{'}, x_2, x_3, ..., x_{i-1}, x_{i+2},...,x_{n}$. 
	
	Or define ${x_1^{'}}= y_{1}+y_{2}+x_{1} > 0$, re-order $\big\{x_i: i \neq 1, i \notin N_1 \cup N_2\big\}, x_1^{'}$
	
	Repeat the above steps till all ${x_i}^{'}$ are postive and finish the proof.
\end{proof}

\begin{proposition}
	Consider a step change added to $\bd s_1$, where ${\bd s_1}^{'}(t) = \bd s_1(t)+P_i U_i$, $t \in [0,T]$. Suppose $P_i$ is positive \footnote{The proof for negative $P_i$ is the same, just change $P_i$ to $\left|P_i\right|$}, the rainflow cycle decomposition results for $\bd s_1$ and ${\bd s_1}^{'}$ are,

	$$\bd S_1: d_1, d_2,..., d_m, ..., d_M\,,$$
	$${\bd S_1}^{'}: {d_1}^{'}, {d_2}^{'},..., {d_n}^{'}, ..., {d_N}^{'}\,,$$
	
	Define $L = max(M,N)$, we could re-write the cycles in $\bd s_1$ and ${\bd s_1}^{'}$ as,
	
	$$\bd s_1: \underbrace{d_1, d_2, ..., d_M, 0, 0,...}_{L}\,,$$
	$${\bd s_1}^{'}: \underbrace{{d_1}^{'}, {d_2}^{'}, ..., {d_N}^{'}, 0, 0,...}_{L}\,,$$	
	
	Define $\Delta d_i$ such that,
	$${d_i}^{'} = d_i + \Delta d_i\,,\forall i = 1,2,...,L$$
	
	The following relations always holds,
	\begin{equation}
	\left|\sum_{i=1}^{L} \Delta d_i\right| \leq P_i\,,
	\end{equation}
	\begin{equation}
	|\Delta d_i| \leq P_i\,,
	\end{equation}
\label{append_c6}
\end{proposition}

\begin{proof}
There exists a small enough $\Delta P$ such that only one cycle depth $d_i$ will change.

$$|\Delta d_i| \leq \Delta P\,,$$
$$-\Delta P \leq \Delta d_i \leq \Delta P\,,$$

Consider $P_i$ as a cumulation of small $\Delta P$, by the principle of integration, we have 
$$-\int \Delta P dp \leq \sum_{i=1}^{L} \Delta d_i \leq \int \Delta P dp\,,$$

Such that,
$$\left|\sum_{i=1}^{L} \Delta d_i\right| \leq P_i$$

$|\Delta d_i| \leq P_i$ holds for the worst case where all cycle depth changes happen at one certain cycle. Therefore, it is trivial to show that $|\Delta d_i| \leq P_i$ hold in all conditions.

Q.E.D.
\end{proof}

Now, by Proposition 1-6, we finish the following proof of Lemma 1.
\begin{proof}
	Let's consider $\bd s^{'}(t) = \lambda \bd s_1(t) + (1-\lambda)P_i U_i, t \in [0,T]$. Then the rainflow cycle decomposition results for $\lambda \bd s_1$ and $\bd s^{'}$ are
	\begin{align*}
	\lambda \bd s1: & \underbrace{\lambda d_1, \lambda d_2, ..., \lambda d_M, 0, 0,...}_{L} \\
	{\bd s_1}^{'}: & \underbrace{{d_1}^{'}, {d_2}^{'}, ..., {d_N}^{'}, 0, 0,...}_{L}
	\end{align*}
Define $\Delta d_i$ such that,
$${d_i}^{'} = \lambda d_i + (1-\lambda) \Delta d_i\,,\forall i = 1,2,...,L$$
\begin{subequations}
	\begin{align}
	\begin{split}
	& f\big(\lambda \bd s_1 + (1-\lambda)P_i U_i\big) \\
	= & \sum_{i=1}^{L} \Phi \big(\lambda d_i + (1-\lambda)\Delta d_i\big) \\
	= & \sum_{i=1}^{l^{+}} \underbrace{\Phi \big(\lambda d_i + (1-\lambda)\Delta d_i\big)}_{\Delta d_i \geq 0}  + \sum_{i=1}^{l^{-}} \underbrace{\Phi \big(\lambda d_i - (1-\lambda)|\Delta d_i|\big)}_{\Delta d_i < 0}\\
	\leq & \sum_{i=1}^{l^{+}} [\lambda \Phi(d_i) + (1-\lambda) \Phi(\Delta d_i)] + \sum_{i=1}^{l^{-}} [\lambda \Phi(d_i) - (1-\lambda) \Phi(|\Delta d_i|)] \\
	\leq & \lambda \sum_{i=1}^{L} \Phi(d_i) + (1-\lambda) \big[\sum_{i=1}^{l^{+}} \Phi(\Delta d_i) - \sum_{i=1}^{l^{-}} \Phi(|\Delta d_i|)\big]
	\end{split}
	\end{align}
\end{subequations}

By Proposition \ref{append_c6}, we have
$$|\sum_{i=1}^{L}  \Delta d_i| \leq P_i\,,$$
$$|\Delta d_i| \leq P_i\,,$$

(1) Assume $\sum_{i=1}^{L} \Delta d_i = P_i$, $|\Delta d_i| \leq P_i$. By Proposition \ref{append_c5},
\begin{small}
\begin{align*}
& f\big(\lambda \bd s_1 + (1-\lambda)P_i U_i\big) \\
\leq & \lambda \sum_{i=1}^{L} \Phi(d_i) + (1-\lambda) \big[\sum_{i=1}^{l^{+}} \Phi(\Delta d_i) - \sum_{i=1}^{l^{-}} \Phi(|\Delta d_i|)\big]\\
\leq  & \lambda \sum_{i=1}^{L} \Phi(d_i) + (1-\lambda) \Phi(\sum_{i=1}^{L}  \Delta d_i)\\
= & \lambda \sum_{i=1}^{L} \Phi(d_i) + (1-\lambda) \Phi(P_i)
\end{align*}
\end{small}

(2) Assume $-P_i \leq \sum_{i=1}^{L} \Delta d_i <P_i$, $|\Delta d_i| \leq P_i$. 

Add some ``virtual cycles'' $d_{L+1}^{'}, d_{L+2}^{'}, ..., d_{L+K}^{'}$ at the end of ${\bd s_1}^{'}$, each $d_{L+i}^{'}$ is positive and satisfies that $|d_{L+i}^{'}| \leq P_i$. So that $\sum_{i=1}^{L+K} \Delta d_i = P_i$, $|\Delta d_i| \leq P_i, \forall i \in [1,2,...,L+K]$. Write 0 at the end of $\lambda \bd s_1$ to achieve the same cycle number.
\begin{align*}
	\lambda \bd s1: & \underbrace{\lambda d_1, \lambda d_2, ..., \lambda d_M, 0, 0, 0,...,0}_{L+K} \\
	\bd s_1^{'}: & \underbrace{{d_1}^{'}, {d_2}^{'}, ..., {d_N}^{'}, 0, 0,...,0, d_{L+1}^{'}, d_{L+2}^{'}, ..., d_{L+K}^{'}}_{L+K}
\end{align*}

\resizebox{.95\linewidth}{!}{
	\begin{minipage}{\linewidth}
\begin{align*}
& f\big(\lambda \bd s_1 + (1-\lambda)P_i U_i\big) \\
\leq &\lambda \sum_{i=1}^{L} \Phi(d_i) + (1-\lambda) \big[\sum_{i=1}^{l^{+}} \Phi(\Delta d_i) - \sum_{i=1}^{l^{-}} \Phi(|\Delta d_i|)\big]\\
<& \lambda \sum_{i=1}^{L} \Phi(d_i) + (1-\lambda) \big[\sum_{i=1}^{l^{+}} \Phi(\Delta d_i) + \sum_{i=L+1}^{L+K} \Phi(\Delta d_i) - \sum_{i=1}^{l^{-}} \Phi(|\Delta d_i|)\big]\\
\leq & \lambda \sum_{i=1}^{L} \Phi(d_i) + (1-\lambda) \Phi(\sum_{i=1}^{L+K}  \Delta d_i)\\
= & \lambda \sum_{i=1}^{L} \Phi(d_i) + (1-\lambda) \Phi(P_i)
\end{align*}
\end{minipage}
}

To sum up, 
\begin{small}
\begin{align}
f\big(\lambda \bd s_1 + (1-\lambda)P_i U_i\big) & \leq \lambda \sum_{i=1}^{L} \Phi(d_i) + (1-\lambda) \Phi(P_i) \nonumber\\
& = \lambda f(\bd s_1) + (1-\lambda) f(P_i U_i)\,, 
\end{align}
\end{small}

where $\lambda \in [0,1]$. Thus, $f(\bd s)$ is convex up to every step change in $\bd s$. 
\end{proof}

Lemma 1 shows that $f(\bd s)$ is convex up to every step change in $\bd s$. Next, we will prove the general rainflow convexity by induction method.

\subsection{General rainflow cycle life loss convexity}
%
%
%
		We will prove the general rainflow convexity by induction.
		
		\emph{I). Initial condiction:} By Lemma 1, $K=1$
		$$f\big(\lambda \bd s_1 + (1-\lambda)\bd s_2\big) \leq\lambda f(\bd s_1) + (1-\lambda) f(\bd s_2)\,, \lambda \in [0,1]$$
		
		II). Suppose that, $f(\bd s)$ is convex up to the sum of $K$ step changes (arranged by time index)
		$$f\big(\lambda \bd s_1 + (1-\lambda) \bd s_2\big) \leq \lambda f(\bd s_1) + (1-\lambda) f(\bd s_2)\,, \lambda \in [0,1], \bd s_1, \bd s_2 \in \mathbb{R}^{K}$$
		
		Then we prove $f(\bd s)$ is convex up to the sum of $K+1$ step changes (see Fig. \ref{fig_induction}),
		$$f\big(\lambda \bd s_1 + (1-\lambda) \bd s_2\big) \leq \lambda f(\bd s_1) + (1-\lambda) f(\bd s_2)\,, \lambda \in [0,1], \bd s_1, \bd s_2 \in \mathbb{R}^{K+1}$$

The following proposition is needed for the proof.
\begin{proposition}\label{append_c7}
		\begin{equation}
		f(\sum_{t=1}^{K}X_t U_t) \geq f(\sum_{t=1}^{i-1} X_t U_t + (X_i+X_{i+1})U_i + \sum_{t=i+2}^{K} X_t U_t)\,,
		\end{equation}
		In other words, the cycle stress cost will reduce if combining adjacent unit changes.
\end{proposition}
\begin{proof}
	Rainflow cycle counting algorithm only considers local extreme points.
	
	I) If $X_i$ and $X_{i+1}$ are the same direction, combining them doesn't affect the value of local extreme points. Therefore the left side cost equals right side cost.
	
	II) If $X_i$ and $X_{i+1}$ are different directions, suppose $X_i$ is negative and $X_{i+1}$ positive (otherwise the same). Time $t=i$ makes a local minimum point.
	\begin{itemize}
	\item Case a: If $|X_{i+1}| \leq |X_{i}|$, combining them will raise the value of local minimum point $i$, thus reducing the depth of cycles which contains $i$. Therefore, the cost after combining is less than the original cost.
	\item Case b: If $|X_{i+1}| > |X_{i}|$, combining them will lead to the removal of local minimum point $i$. 
	
	In one case, if $X_{i-1}$ and $X_i$ are the same direction, time $t=i-1$ will make a local minimum point taking the place of time $t=i$. Therefore, the magnitude of the local minimum point decreases, similar to case a, the total cost after combining is less than the original cost.
	
	In the other case, if $X_{i-1}$ and $X_i$ are different directions, we lose a full cycle with depth $|X_i|$ fater combining. So the cost after combining is also less than the original.
	\end{itemize}
	To sum up, the cycle stress cost will reduce if combining adjacent unit changes.
\end{proof}
	
Recall the step function decomposition results for $\bd s_1$, $\bd s_2$ and $\lambda \bd s_1 + (1-\lambda) \bd s_2$.
\begin{align*}
& \bd s_1 = \sum_{i =1}^{T} Q_i U_i \,,\\
& \bd s_2 = \sum_{i =1}^{T} P_i U_i \,, \\
& \lambda \bd s_1 + (1-\lambda) \bd s_2  = \sum_{i =1}^{T} X_i U_i\,, 
\end{align*}
There are three cases when we go from $T=K$ to $T=K+1$, classfied by the value and symbols of $X_{K}$, $X_{K+1}$.

\noindent \textbf{Case 1:} $X_{K}$ and $X_{K+1}$ are same direction. 
\begin{proof}
If $X_{K+1}$ and $X_{K}$ are same direction, we could move $X_{K+1}$ to the previous step without affecting the total cost $f(\bd s_1 + (1-\lambda) \bd s_2)$. Then we prove the $K+1$ convexity by applying Proposition \ref{append_c7}. 
\begin{align*}
	& f\left(\lambda \mathbf{s_1} + (1-\lambda) \mathbf{s_2}\right) \nonumber\\
	= & f\left(\lambda \mathbf{s_1}^{K} + (1-\lambda) \mathbf{s_2}^{K} + X_{K+1}U_{K}\right) \nonumber\\
	= & f\left\{\lambda \mathbf{s_1}^{K} + (1-\lambda) \mathbf{s_2}^{K} + [\lambda Q_{K+1} + (1-\lambda) P_{K+1}] U_{K}\right\}\nonumber\\	
	\leq & \lambda f\left(\mathbf{s_1}^{K}+ Q_{K+1}U_{K}\right) + (1-\lambda)f\left(\mathbf{s_2}^{K}+ P_{K+1}U_{K}\right) \nonumber\\
	\leq & \lambda f(\mathbf{s_1}) + (1-\lambda)f(\mathbf{s_2}) \ (\text{by Lemma \ref{lemma:single_step_proof}})
\end{align*}
\end{proof}

\noindent \textbf{Case 2:} $X_{K}$ and $X_{K+1}$ are different directions, with $|X_{K}| \geq |X_{K+1}|$. 
\begin{proof}
In this case, the last step $X_{K+1}$ could be seperated out from the previous SoC profile. Therefore
\begin{align*}
	& f\left(\lambda \mathbf{s_1} + (1-\lambda) \mathbf{s_2}\right) & \nonumber\\
	= & f\left(\lambda \mathbf{s_1} + (1-\lambda) \mathbf{s_2}^{K}\right) + \Phi\left(X_{K+1} U_{K+1}\right) \nonumber\\
	\leq & \lambda f(\mathbf{s_1}^{K}) + (1-\lambda)f(\mathbf{s_2}^{K}) + \Phi\left[\lambda Q_{K+1} U_{K+1} + (1-\lambda) P_{K+1}U_{K+1}\right] & \nonumber\\
	\leq & \lambda \big[f(\mathbf{s_1}^{K}) + \Phi(Q_{K+1}U_{K+1})\big] + (1-\lambda)\big[f(\mathbf{s_2}^{K})  + \Phi(P_{K+1} U_{K+1})\big] & \nonumber\\
	\leq &\lambda f(\mathbf{s_1}) + (1-\lambda)f(\mathbf{s_2}) &
\end{align*}
\end{proof}
\noindent \textbf{Case 3:} $X_{K}$ and $X_{K+1}$ are different directions, with $|X_{K}| < |X_{K+1}|$. 
\begin{proof}
In such condition, $X_{K+1}$ may not easily seperated out from previous SoC. It further contains three cases.
\begin{itemize}
	\item $X_{K-1}$ and $X_{K}$ are the same direction. We could use the same ``trick" in Case 1 to combine step $K-1$ and $K$. Proof is trivial for this case.
	\item $X_{K-1}$ and $X_{K}$ are diffrent directions, $X_{K}$ and $X_{K+1}$ form a cycle that is \emph{separate} from the rest of the signal (eg. it is the deepest cycle). We can separate $X_{K+1}$ out, and proof will be similar to Case 2. 
	\item $X_{K-1}$ and $X_{K}$ are diffrent directions, $X_{K}$ and $X_{K+1}$ do not form a separate cycle. This condition is the most complicated case, since it's hard to move $X_{K+1}$ to the previous step, or seperate it out. For example, in Fig. \ref{fig_append_induction}, $\delta_{K+1}$ was counted as a part of the previous charging cycle while $\Delta_{K+1}$ forms a closed full cycle together with $X_{K}$ by Rainflow counting algorithm. Therefore, we need to look into $Q_K$, $Q_{K+1}$, $P_K$, $P_{K+1}$ in order to show the $K+1$ step convexity. 
	\begin{figure}[ht]
		\centering
		\includegraphics[width= 2.5 in]{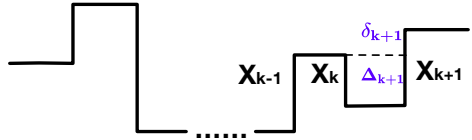}
		\caption{$X_{K}$ and $X_{K+1}$ are different directions, with $|X_{K}| < |X_{K+1}|$}
		\label{fig_append_induction}
	\end{figure}
	
	There are further four sub-cases (Fig. \ref{fig_append_4case}), for simplicity we only consider the cost of charging cycles. Showing convexity for each sub-case finishes the overall convexity proof. 
	\begin{figure}[ht]
	\centering
	\includegraphics[width= 0.6\columnwidth, height = 0.6\columnwidth]{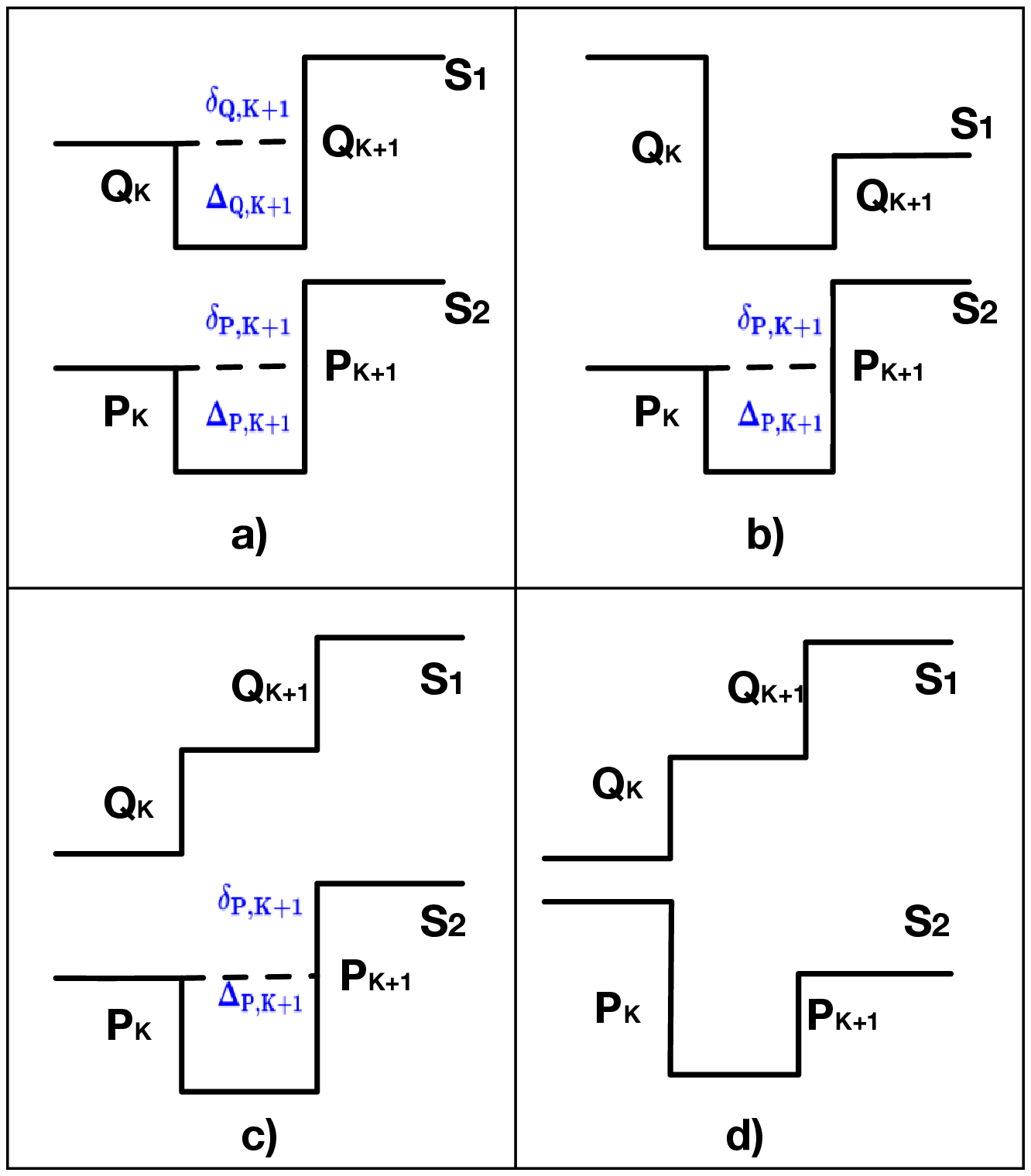}
	\caption{Four cases of $Q_K$, $Q_{K+1}$, $P_K$, $P_{P+1}$}
	\label{fig_append_4case}
	\end{figure}
	    	    
	    Case a) 
	    \begin{align*}
	    & f(\lambda \bd s_1 + (1-\lambda) \bd s_2) \nonumber\\
	    = & f\left(\lambda \bd s_1^{K} + (1-\lambda) \bd s_2^{K} + X_{K+1}U_{K+1}\right) \nonumber\\
	    = & f(\lambda \bd s_1^{K} + (1-\lambda)\bd s_2^{K} + X_{K+1}U_{K}) + \Phi(\Delta_{K+1}) \nonumber\\
	    \leq & \lambda f\big(\bd s_1^{K}+Q_{K+1}U_{K}\big) + (1-\lambda) f\big(\bd s_2^{K}+P_{K+1}U_{K}\big) \nonumber\\
	    & + \Phi\big(\lambda \Delta_{Q,K+1}+(1-\lambda) \Delta_{P,K+1}\big) \nonumber\\
	    \leq & \lambda \big[f(\bd s_1^{K}+Q_{K+1}U_{K})+\Phi(\Delta_{Q,K+1})\big] \nonumber\\
	    & + (1-\lambda) \big[f(\bd s_2^{K}+P_{K+1}U_{K})+\Phi(\Delta_{P,K+1})\big] \nonumber\\
	    = & \lambda f(\bd s_1) + (1-\lambda) f(\bd s_2)
	    \end{align*}	
	    
	    Case c) 
	    \begin{align*}
	    & f(\lambda \bd s_1 + (1-\lambda) \bd s_2) \nonumber\\
	    = & f(\lambda \bd s_1^{K} + (1-\lambda) \bd s_2^{K} + X_{K+1}U_{K+1}) \nonumber\\
	    = & f(\lambda \bd s_1^{K} + (1-\lambda)\bd s_2^{K} + X_{K+1}U_{K}) + \Phi(\Delta_{K+1}) \nonumber\\
	    \leq & \lambda f\big(\bd s_1^{K}+Q_{K+1}U_{K}\big) + (1-\lambda) f\big(\bd s_2^{K}+P_{K+1}U_{K}\big) + \Phi(\Delta_{K+1}) \nonumber\\
	    \leq & \lambda f\big(\bd s_1^{K}+Q_{K+1}U_{K}\big) + (1-\lambda) f\big(\bd s_2^{K}+P_{K+1}U_{K}\big) \nonumber\\
	    & + \Phi\big[(1-\lambda) \Delta_{P,K+1}-\lambda Q_{K+1}\big] \nonumber\\
	    \leq & \lambda f\big(\bd s_1^{K}+Q_{K+1}U_{K}\big) + (1-\lambda) f\big(\bd s_2^{K}+P_{K+1}U_{K}\big) \nonumber\\
	    & + \Phi\big[(1-\lambda) \Delta_{P, K+1}\big] \nonumber\\
	    \leq & \lambda f(\bd s_1) + (1-\lambda) f(\bd s_2)
	    \end{align*}	

	    Then we try to finish the proof for case b) and d). First, note that b) implies d).  To show this, for case d), define $\hat{\bd s_1} = \sum_{t=1}^{K+1} \hat{Q_t} U_t$ as a modified version of $\bd s_1$, where $\hat{Q_t} = Q_t$ for $t=1,...,K-1$, $\hat{Q}_{K} = 0$, $\hat{Q}_{K+1} = Q_{K}+Q_{K+1}$. We have $f(\hat{\bd s}_1) = f(\bd s_1)$. We also have $f(\lambda \hat{\bd s}_1+(1-\lambda)\bd s_2) \geq f(\lambda \bd s_1+(1-\lambda)\bd s_2)$ because of the decreasing signal at $K$ for $\bd s_2$. Thus,
	    \begin{align*}
	    f(\lambda \bd s_1 + (1-\lambda)\bd s_2) &\leq f(\lambda \hat{\bd s}_1+(1-\lambda)\bd s_2) \nonumber\\
	   &  \stackrel{i)}\leq \lambda f(\hat{\bd s_1}) + (1-\lambda)f(\bd s_2) \nonumber\\
	   & = \lambda f(\bd s_1) +(1-\lambda)f(\bd s_2)
	    \end{align*}
    
    where i) follows from assuming b) is true and letting $\Delta_{P,K+1} =0$ and reversing the label of P and Q. Therefore we only need to prove case b). 
    
    Case b) contains two different circumstances in terms of $P_{K}$ and $P_{K+1}$. We need the following proposition for the two cases. To not use too many negative signs, we denote $\bar{Q}_t = -Q_t$, $\bar{P}_t = -P_t$.
	 
	 \begin{proposition}
	 	Let $g$ be a convex increasing function and given numbers $a > b >0$. Then $g(a)+g(b) \geq g(a-\delta) + g(b+\delta)$ if $\delta >0$ and $b+\delta < a$,
	 \end{proposition}	  
	 
	 \begin{proof}
	 	Define $h(x) = g(x)-g(x-\delta)$ and $x, x-\delta \geq 0$. We have
	 	$$h^{'}(x) = g^{'}(x)-g^{'}(x-\delta) \geq 0\,,$$
	 	because $g$ is convex and $x > x-\delta$. Therefore, $h(\cdot)$ is an increasing function, $\forall a > b+\delta$,
	 	\begin{align*}
	 	h(a) & \geq h(b+\delta)\\
	 	g(a) - g(a-\delta) &\geq g(b+\delta) - g(b)
	 	\end{align*}
	 	Moving $g(b)$ to the left side and $g(a-\delta)$ to the right side of the inequation finish the proof.
	 \end{proof}
	 
	 Case b) \textcircled{1} $P_{K}$, $P_{K+1}$ do not form a cycle that is separate from the rest of $\bd S_2$. 
	 \begin{align*}
	  & f(\lambda \bd s_1 + (1-\lambda)\bd s_2) \nonumber\\
	  = & f(\lambda \bd s_1^{K} + (1-\lambda)\bd s_2^{K} + X_{K+1}U_{K}) + \Phi(\Delta_{K+1}) \nonumber\\
	  = & f(\lambda \bd s_1^{K} + (1-\lambda) \bd s_2^{K} + \lambda Q_{K+1}U_{K} + (1-\lambda) P_{K+1}U_{K}) + f(\lambda \bar{Q}_{K}U_{K+1} + (1-\lambda)  \bar{P}_{K}U_{K+1}) \nonumber\\
	  = & f\left(\lambda \bd s_1^{K} + \lambda \bar{Q}_{K}U_{K} - \lambda \bar{Q}_{K}U_{K} + (1-\lambda) \bd s_2^{K} + \lambda Q_{K+1} U_{K}  + (1-\lambda) P_{K+1}U_{K}\right) \nonumber\\
	  & +f\left(\lambda Q_{K+1}U_{K+1} - \lambda Q_{K+1}U_{K+1} + \lambda \bar{Q}_{K}U_{K+1} + (1-\lambda) \bar{P}_{K}U_{K+1}\right) \nonumber\\
	  = & f\Big(\lambda \bd s_1^{K} + \lambda \bar{Q}_{K}U_{K}  +(1-\lambda)\big(\bd s_2^{K}+ P_{K+1}U_{K} + \frac{\lambda}{1-\lambda}(Q_{K+1}U_{K}- \bar{Q}_{K}U_{K}) \big)\Big) \nonumber\\
	  &+ f \Big(\lambda Q_{K+1}U_{K+1} + (1-\lambda) \big( \bar{P}_{K}U_{K+1}+\frac{\lambda}{1-\lambda}(\bar{Q}_{K}U_{K+1}- Q_{K+1}U_{K+1} )\big)\Big) \nonumber\\
	  \leq & \lambda f(\bd s_1^{K}+\bar{Q}_{K}U_{K}) + \lambda f(Q_{K+1}U_{K+1} ) \nonumber\\
	  &+(1-\lambda)f\big(\bd s_2^{K}+\big(P_{K+1}-\frac{\lambda}{1-\lambda}(\bar{Q}_{K}-Q_{K+1})\big)U_{K}\big)\nonumber\\
	  &+(1-\lambda)f\big(\big(\bar{P}_{K}+\frac{\lambda}{1-\lambda}(\bar{Q}_{K}-Q_{K+1})\big)U_{K+1}\big) 
	  \end{align*}
	  
	  Only considering charging cycles, the first line is the cost of $\lambda f(\bd s_1)$. Now we show, 
	  \begin{align*}
	  f\big(\bd s_2^{K}+\big(P_{K+1}-\frac{\lambda}{1-\lambda}(\bar{Q}_{K}-Q_{K+1})\big)U_{K}\big)\nonumber\\
	  +f\big(\big(\bar{P}_{K}+\frac{\lambda}{1-\lambda}(\bar{Q}_{K}-Q_{K+1})\big)U_{K+1}\big) \leq f(\bd s_2)\,,
	  \end{align*}
	  
	  We can write out the cost of charging cycles in $\bd s_2$ as $f(\bd s_2)$,
	  
	  \begin{equation*}
	  f(\bd s_2) = \sum_{i}^{N-1} \Phi(Z_i) + \Phi(\bar{P}_{K})\,,
	  \end{equation*}
	  
	  where $Z_{N-1}$ is the cycle that $\bd s_2(K+1)$ is assigned to. 
	  
	  By assumption that $P_{K}$ and $P_{K+1}$ do not form a separate cycle, so that $Z_{N-1} \geq P_{K+1}$. By assumption, $\frac{\lambda}{1-\lambda}(\bar{Q}_{K}-Q_{K+1}) > 0$. Let  $\delta=\frac{\lambda}{1-\lambda}(\bar{Q}_{K}-Q_{K+1})$, and since $P_{K+1}+Q_{K+1} \geq \bar{P}_{K}+\bar{Q}_{K}$ in case b), $\bar{P}_{K}+\delta \leq P_{K+1} \leq Z_{N-1}$. Therefore applying Proposition 8, $a = Z_{N-1}$ and $b = \bar{P}_{K}$, $\delta=\frac{\lambda}{1-\lambda}(\bar{Q}_{K}-Q_{K+1})$, we have the desired result. 

	  Case b) \textcircled{2} The other case is that $P_{K}$, $P_{K+1}$ form a cycle that is separate from the rest of $\bd s_2$. Similar as \textcircled{1}, we need to show
	  \begin{align*}
	    f\big(\bd s_2^{K}+\big(P_{K+1}-\frac{\lambda}{1-\lambda}(\bar{Q}_{K}-Q_{K+1})\big)U_{K}\big) \nonumber\\
	    +f\big(\big(\bar{P}_{K}+\frac{\lambda}{1-\lambda}(\bar{Q}_{K}-Q_{K+1})\big)U_{K+1}\big) \leq f(\bd s_2) 
	   \end{align*}
	   
	   Since $P_{K+1} = \bar{P}_{K}+\delta_{P, K+1}$, we re-write the above inequation as,
	   \begin{align*}
	   f\big(\bd s_2^{K}+\big(\bar{P}_{K}+\delta_{P, K+1}-\frac{\lambda}{1-\lambda}(\bar{Q}_{K}-Q_{K+1})\big)U_{K}\big)\nonumber\\ +f\big(\big(P_{K+1}-\delta_{P, K+1} + \frac{\lambda}{1-\lambda}(\bar{Q}_{K}-Q_{K+1})\big)U_{K+1}\big) \leq f(\bd s_2)\,,
	   \end{align*}
	  
	  Denote $\delta = \delta_{P, K+1}-\frac{\lambda}{1-\lambda}(\bar{Q}_{K}-Q_{K+1}) > 0$
	  
	  \begin{equation*}
	  f(\bd s_2) = \sum_{i}^{N-2} \Phi(Z_i) + \Phi (Z_{N-1})+ \Phi(P_{K+1})\,,
	  \end{equation*}
	  
	  $Z_{N-1}$ is the deepest charging cycle with ending SoC equals the $P_{K}$'s starting SoC, and $Z_{N-1} \leq P_{K+1}$ since $P_{K+1}$ forms a separate cycle. Applying Proposition 8 by setting $a = P_{K+1}$, $b = Z_{N-1}$, we have the desired results.
	  \end{itemize}
\end{proof}
\end{document}